\documentclass[11pt]{amsart}
\usepackage[centertags]{amsmath}
\usepackage{amsfonts}
\usepackage{amsthm}
\usepackage{newlfont}
\usepackage{amscd}
\usepackage{amsgen}
\usepackage{amssymb}
\usepackage{longtable}
\usepackage{listings}
\usepackage{fullpage}
\newlength{\defbaselineskip} 

 \theoremstyle{plain} \newtheorem{thm}{Theorem}[section]
\newtheorem{prob}[thm]{Problem} \newtheorem{lemm}[thm]{Lemma}
\newtheorem{prop}[thm]{Proposition}
\theoremstyle{definition}
\newtheorem{defi}[thm]{Definition}
\newtheorem{rem}[thm]{Remark}

\newtheorem{cor}[thm]{Corollary}
 \numberwithin{equation}{section}
\numberwithin{equation}{section}

\newcommand\PP{{\mathbb{P}}}

          \newcommand\oo{{\mathcal O}}
\newcommand\E{{\mathcal{E}}}

\DeclareMathOperator{\Pic}{Pic}

 \DeclareMathOperator{\rk}{rk}
  
 \DeclareMathOperator{\Ext}{Ext}

\DeclareMathOperator{\Pf}{Pf}
\theoremstyle{definition} 
\newtheorem*{Claim}{Claim}
\title{Calabi--Yau threefolds in $\mathbb{P}^6$}
\author{Grzegorz Kapustka, Micha\l \ Kapustka}
\begin{document}

\begin{abstract} We study the geometry of $3$-codimensional smooth subvarieties of
the complex projective space. In particular, we classify all quasi-Buchsbaum Calabi--Yau threefolds in projective $6$-space. Moreover, we prove that this classification includes all 
Calabi--Yau threefolds contained in a possibly singular 5-dimensional quadric as well as all Calabi--Yau threefolds of degree at most $14$ in $\mathbb{P}^6$.
\end{abstract}

 \maketitle
 \section{Introduction}

It is conjectured that, when $2n\geq N$, there is a finite number of
smooth families of smooth $n$-dimensional subvarieties of $\mathbb{P}^N$ that
are not of general type.
This conjecture was inspired by \cite{EP} where the statement was formulated and proven in the case of surfaces in
$\mathbb{P}^4$. In \cite{BOS}, the conjecture was proven in the case of threefolds in $\mathbb{P}^5$.
Moreover,
Schneider in \cite{Sch} proved that the statement is true when $2n\geq N+2$.
In this context, it is a natural problem to classify families of
smooth $n$-folds of small degree in $\mathbb{P}^N$ for chosen $n,N\in \mathbb{N}$
satisfying $2n\geq N$.
In the case of codimension $2$ subvarieties, this problem was addressed
by many authors (see \cite{BSS}, \cite{EFG}, \cite{DES},\cite{AR}).

The next step is to study codimension $3$ subvarieties in $\mathbb{P}^6$. In
this case, standard tools such as the Barth-Lefschetz theorem do not apply.
However, some
general structure theorems were recently developed.
We say that a submanifolds $X\subset \mathbb{P}^n$ is subcanonical when
$\omega_X=\mathcal{O}_X(k)$ for some $k\in\mathbb{Z}$.
A codimension $3$ submanifold $X$ is called \textit{Pfaffian} if it is the first nonzero
degeneracy locus of a skew-symmetric morphism of vector bundles of odd rank
$E^{\ast}(-t)\xrightarrow{\varphi} E$ where $t\in\mathbb{Z}$.
In this case, $X$ is given locally by the vanishing of $2u\times 2u$
Pfaffians of an alternating map $\varphi$ from the vector bundle $E$ of odd rank
$2u+1$ to its twisted dual. More precisely, if $X$ is Pfaffian then we have:
\begin{equation}\label{pfaffian sequence}
 0\to\mathcal{O}_{\mathbb{P}^n}(-2s-t)\to E^{\ast}(-s-t)\xrightarrow{\varphi}
E(-s)\to\mathcal{I}_X\to 0
\end{equation}
where $s=c_1(E)+ut$. Moreover, from \cite[\textsection 3]{O}, we have in this case
\begin{equation}\label{pfaffian adjunction} \omega_X=\mathcal{O}_X(t+2s-n-1).
\end{equation}
Since the choice of an alternating map $\varphi$ is equivalent to the choice of a section $\sigma\in H^0(\bigwedge^2E(t))$,
we will use the notation $\Pf(\sigma)$ for the variety described by the Pfaffians of the map corresponding to $\sigma$.

Answering a question of Okonek (see \cite{O}), Walter, in \cite{W}, proved that if $n$ is not
divisible by $4$ then a locally Gorenstein codimension $3$ submanifold of
$\mathbb{P}^{n+3}$ is Pfaffian if and only if it is subcanonical. In the case when $n=4k$,
the last statement is not true, however, there is another structure theorem
(see \cite{EPW}).

The non-general type subcanonical threefolds in $\mathbb{P}^6$
are either well understood Fano threefolds or threefolds with trivial canonical class. A very natural class of varieties among varieties with trivial canonical class are
Calabi--Yau threefolds i.e.~smooth threefolds $X$ with $K_X=0$ and
$H^1(X,\mathcal{O}_X)=0$. In the paper, we shall sometimes also consider \textit{singular Calabi--Yau} threefolds by which we mean complex projective threefolds with Gorenstein singularities, $\omega_X=0$ and
with $h^1(\oo_X)=0$.

For Calabi--Yau threefolds the theory of Pfaffians is more specific. For instance,
Schreyer, following \cite{W} shows that if $X$ is Pfaffian,
$h^i(\mathcal{O}_X)=0$ for $0<i< \dim X$ and
$E$ is any vector bundle such that there exists $\sigma\in H^0(\bigwedge^2E(t))$ with  $X=\Pf(\sigma)$ then, keeping the notation above,  $E(-s)$ appears as a sum of 
the sheafified first syzygy module $Syz^1(HR(X))$ of the Hartshorne--Rao module $HR(X)=\oplus_{k=1}^{N} H^1(\mathcal{I}_X(k))$
(seen as a $\mathfrak{S}=\mathbb{C}[x_0,\dots,x_6]$ module) and a sum of line bundles. 
If we add the assumption that $X$ has trivial canonical class then, by considering an appropriate twist, we can choose a bundle $E$ such that $t=1$ and $s=3$. More precisely, if $X$ is a Calabi--Yau threefold then
there exists a bundle $E$ of rank $2u+1$ such that $3=c_1(E)+u$ and $X=\Pf(\sigma)$ for some $\sigma\in H^0(\bigwedge^2E(1))$. Moreover, if we denote by $M$ the Hartshorne--Rao module of $X$ with gradation shifted by 3 then
the chosen bundle $E$ is obtained as a sum of $Syz^1(M)$ with a sum of line bundles.

Let us point out that all threefolds can be smoothly projected to $\mathbb{P}^7$.
It is, moreover, known from \cite[rem.~11]{BC} that Calabi--Yau threefolds embedded in $\mathbb{P}^5$ are complete
intersections; either of two cubics, or of a quadric and a quartic, or
of a quintic and
a hyperplane. Having this in mind, we study nondegenerate Calabi--Yau threefolds i.e. such Calabi--Yau threefolds which are not contained in any hyperplane.

Nondegenerate Calabi--Yau threefolds in $\mathbb{P}^6$ were already studied in \cite{Ro}, \cite{TO},
\cite{Be}, \cite{K}, and \cite{KK1}, where examples of degree $12\leq d \leq 17$ were constructed.
It is not hard to see that the degree of such a threefold is bounded between $11$
and $41$ (see Corollary \ref{ograniczenie na stopien}) but we expect a sharper
bound (see \cite{KK}). Okonek proposed the
following problem:
\begin{prob} \label{prob classification} Classify the Calabi--Yau threefolds in $\mathbb{P}^6$.
\end{prob}

The central result of the  paper is a full classification of quasi-Buchsbaum Calabi--Yau threefolds in $\PP^6$ i.e.~Calabi--Yau threefolds in $\PP^6$ 
such that their higher cohomology modules have trivial structure (see Definition \ref{def ACM i QB}).
The classification is given in Theorem \ref{Buchsbaum CY}. The proof that this classification includes all Calabi--Yau
threefolds in $\mathbb{P}^6$ of degree $d\leq 14$ and a classification of
Calabi--Yau threefolds contained in 5-dimensional quadrics is our main result.
The classification is given by providing a list of vector bundles $\{E_i\}_{i\in I}$ such that the considered Calabi--Yau threefolds are exactly the smooth 
threefolds which appear as Pfaffians $\Pf(\sigma)$ for some $\sigma\in H^0(\bigwedge^2E_i(1))$ and $i\in I$.
Let us point out that our list contains two distinct vector bundles corresponding to degree $14$ Calabi--Yau threefolds.

In Section \ref{sec preliminaries}, we prove basic general results concerning the classification of Calabi--Yau threefolds in $\mathbb{P}^6$. In particular, we observe that
a Calabi--Yau threefolds in $\mathbb{P}^6$ must be linearly normal. We, moreover, prove the finiteness of the classification problem \ref{prob classification}.

Theorem \ref{Buchsbaum CY} is the main theorem of Section \ref{CY of decomposable bundles}. It presents the classification of Calabi--Yau threefolds that are quasi-Buchsbaum.
As a consequence, we find
that the examples that are arithmetically Cohen-Macaulay (see Definition \ref{def ACM i QB}) are of degrees $12\leq d\leq 14$.

In Section \ref{sec deg CY in quadrics}, we classify Calabi--Yau threefolds contained in quadrics in terms of degree. More precisely, we prove the following theorem.

\begin{thm} \label{Main theorem} If  $(X,Q^r_5)$ is a pair consisting of a nondegenerate
Calabi--Yau threefold $X\subset \mathbb{P}^6$ of degree $d_X$ and
a 5-dimensional quadric $Q^r_5$ of corank $r$ in $\mathbb{P}^6$ such that $X\subset
Q^r_5$, then $r\leq 2$ and $12\leq d_X\leq 14$.
\end{thm}

For the proof, we consider case by case the possible coranks of the quadrics
containing the
Calabi--Yau threefolds. For low corank quadrics (i.e. $r=0,1,2$) we consider
hyperplane sections of our Calabi--Yau
threefold and so work with canonically embedded surfaces of general type;
see \cite{Ca} for more information on such surfaces.
For instance, on a smooth quadric in
$\mathbb{P}^5$ containing a canonically embedded surface of general type $S$, we
can apply the double point formula to get
the bound for the degree $d_S$ of $S$ to be $12 \leq d_S\leq 14$. Inputting an
additional assumption on $S$, stating that it is a section of some Calabi--Yau
threefold $X$
contained in a smooth quadric in $\mathbb{P}^6$, leads to the result
$d_X=d_S=12$ or $d_X=d_S=14$.
For canonically embedded surfaces of general type contained in quadrics of rank 
$5$ in $\mathbb{P}^5$,
we obtain the same bound $12 \leq d\leq 14$ working on the resolution of this
quadric. The latter resolution is the projectivization of a vector bundle of rank $2$.
The last step is the proof that there are no Calabi--Yau threefolds contained in
quadrics of corank $\geq 3$ in $\mathbb{P}^6$.
It is worth noticing that there is no similar result for canonically embedded
surfaces
of general type contained in quadrics of corank $\geq 2$ in $\mathbb{P}^5$.
In particular, we present, in Propositions \ref{CY stopnia 11}, \ref{Cy stopnia 15} examples of nodal Calabi--Yau threefolds
contained in quadrics of rank $4$ which have degree $11$ and $15$. Their general hyperplane sections are canonical surfaces of respective degrees $11$ and $15$  which are contained in 4-dimensional quadrics of rank 4. 

The classification of Calabi--Yau threefolds of degree $d\leq 14$ in $\mathbb{P}^6$ in terms of vector bundles associated to them by the Pfaffian construction
is completed in Sections \ref{sec classification} and \ref{sec class degree 14}. More precisely, we prove that all Calabi--Yau threefolds of degree at most  $14$ in $\mathbb{P}^6$ are quasi-Buchsbaum
and use the classification of the latter threefolds contained in Section \ref{CY of decomposable bundles} (see Theorem \ref{Buchsbaum CY}).

By Theorem \ref{Main theorem}, the classification of Calabi--Yau threefolds of degree $d\leq 14$ in $\mathbb{P}^6$ provides also a classification of all Calabi--Yau threefolds contained in 5-dimensional quadrics.
We, moreover, observe that there are two types of Calabi--Yau threefolds of degree $d=14$. Calabi--Yau threefolds of the first type are not contained in any quadric whereas Calabi--Yau
threefolds of the second type are.

Finally, in Section \ref{sec examples}, we perform a classification of Calabi--Yau threefolds of degree $d\leq 14$ in $\mathbb{P}^6$ up to deformation. Since this type of classification
is weaker than the classification in terms of vector bundles
and stronger than the classification by degree, the only remaining ingredient is the proof that there is a unique maximal flat family of Calabi--Yau threefolds of degree 14.
Throughout the paper we study three families of Calabi--Yau threefolds of degree 14.
The first is the family $\mathfrak{C}_{14}$ of degree $14$ Calabi--Yau threefolds
contained in a smooth quadric $Q^0_5$. To define it we think of the smooth 5-dimensional quadric $Q^0_5$ as a homogenous variety with respect to the standard action of the simple Lie group $\mathbb{G}_2$.
Then $Q^0_5$ admits a natural bundle $\mathcal{C}$ called a \emph{Cayley bundle} which is homogeneous with respect to this action.
The family $\mathfrak{C}_{14}$ is the family of all smooth threefolds appearing as zero loci of sections of a twist $\mathcal{C}(3)$. 
To confirm that the family is not complete, we compute that these threefolds have more deformations then obtained by varying
the section of $\mathcal{C}(3)$. In fact, by Corollary \ref{stopnia 14 w kwadryce}, we deduce that $\mathfrak{C}_{14}$ is part of a larger family
$\mathfrak{B}_{14}$ of threefolds given as Pfaffian varieties associated to the bundle $E=\Omega^1_{\mathbb{P}^6}(1)\oplus \mathcal{O}_{\PP^6}(1)$.
Then we prove a technical result (Proposition \ref{degenerations of pfaffians}) on deformation of Pfaffian varieties implying
that any threefold $B_{14}\in \mathfrak{B}_{14}$ appears as a smooth degeneration of the family
$\mathfrak{T}_{14}$ of Calabi--Yau threefolds defined by $6\times 6$ Pfaffians of alternating $7\times 7$
matrices of linear forms. This proves that all families of Calabi--Yau threefolds of degree 14 which appear in the classification of Section \ref{sec class degree 14} are in the same component of the Hilbert scheme.
\begin{cor}
 There is one family of Calabi--Yau threefolds in $\PP^6$ in each degree $d\leq 14$.
\end{cor}

Families of threefolds $\mathfrak{B}_{14}$ and $\mathfrak{B}_{15}$ of degrees 14 and 15 respectively that we consider in this paper were already
constructed in \cite{Be}; however, our results stay in contradiction
with \cite[Prop 4.3.]{Be} and the results in subsections 4.2.2 and
4.2.3 therein.
In particular, we prove that both the examples of degree $14$ and $15$ are flat
deformations of families of Calabi-Yau threefolds constructed in
\cite{TO}. This means that if to each member of a family of deformations of
Calabi-Yau threefolds in $\mathbb{P}^6$ we associate the minimal degree
of hypersurfaces containing it then, unlike in the case of complete intersections, this number can drop for special members.
 We, moreover, prove that the examples of degree $15$ constructed in \cite{Be} are not smooth but admits
three ordinary double points.
\section*{Acknowledgments} We would like to thank Ch.~Okonek for initiating the project, motivation,
and help. We would like to thank L.~Gruson and C.~Peskine for mathematical inspiration.
We would also like to thank
A.~Bertin, J.~Buczynski, S.~Cynk, D.~Faenzi, P.~Pragacz for helpful comments and the referee for helpful remarks and suggestions. The first author was supported by MNSiW, N N201 414539, the second by the Forschungskredit of the University of Zurich and Iuventus Nr IP2011 005071 ``Uklady linii na zespolonych rozmaitosciach kontaktowych oraz uogolnienia''.
\section{Preliminaries}\label{sec preliminaries}
Let us first discuss some general properties of Calabi--Yau threefolds embedded
in $\mathbb{P}^6$. We call such a threefold \textit{nondegenerate} if it is not contained in any hyperplane.
The degenerate Calabi--Yau threefolds (those which are not nondegenerate) in $\PP^6$ are known to be complete intersections either $X_{3,3}\subset \PP^5\subset \PP^6$ or $X_{2,4}\subset \PP^5 \subset \PP^6$ or $X_5\subset \PP^4\subset \PP^6$ (see \cite[rem.~11]{BC}).
\begin{prop}\label{pro linearly normal}
 Let $X\subset \mathbb{P}^6$ be a Calabi--Yau threefold; then $X$ is
linearly normal i.e. the natural restriction map
$H^0(\mathcal{O}_{\mathbb{P}^6}(1))\to H^0(\mathcal{O}_X(1))$ is surjective.
\end{prop}
\begin{proof} It follows from \cite[thm.~2.1]{F} that there are only three families of
non-linearly normal threefolds in $\mathbb{P}^6$.
These families have degrees $6,7$ and $8$ respectively and cannot be Calabi--Yau
threefolds.
\end{proof}

Next, we show an a priori bound on the degree of the Calabi--Yau threefolds
contained in $\mathbb{P}^6$.
\begin{cor}\label{ograniczenie na stopien}
 The degree $d$ of a nondegenerate Calabi--Yau threefold $X\subset \mathbb{P}^6$ is bounded between
$11 \leq d\leq 41$
\end{cor}
\begin{proof} Observe that a generic hyperplane section $S\subset X$ is a
canonically embedded surface of general type.
It was already remarked in \cite{TO} that using the Castelnuovo inequality for
surfaces of general type we can deduce that $d\geq 11$.
Next, from the Riemann--Roch theorem for line bundles on $X$ and Proposition \ref{pro linearly normal}, we deduce that
$7=h^0(\mathcal{O}_X(1))=\chi(\mathcal{O}_X(1))=\frac{1}{12}S.c_2(X)+\frac{1}{6} d$. It is a classical result on Calabi--Yau threefolds, contained in \cite{My},
that $H.c_2(X)\geq 0$ for every ample divisor $H$ on $X$. 
Thus we infer $7 \leq \frac{1}{6}d$.

Moreover, from \cite{My} we also know that, for any ample divisor $H$, we have $H.c_2(X)= 0$ if and only if $X$ is a finite \'etale quotient of an abelian threefold (this implies in particular $c_2(X)=0$).
Let us now show that $d=42$ is impossible. By the above, in this case, it is enough to consider Calabi--Yau threefolds with trivial $c_2(X)$. Those were classified in \cite[thm.~0.1]{OS}.
There are two possibilities and in each of them we have $\chi_{top}(X)=0$. On the other hand, from the double point formula (cf.~\cite{TO}) we get
\begin{equation}\label{dpf} \chi_{top}(X) = -d^2+49d-588=-(42)^2+49\cdot 42-588\neq 0.\end{equation}
We thus obtain a contradiction proving that $d\neq 42$ and in consequence $d\leq 41$.
\end{proof}
\begin{rem}
It is a natural problem to ask if other smooth threefolds with $K_X=0$ (i.e.~without assuming that $h^1(\mathcal{O}_X)=0$) can be embedded in $\PP^6$. Note that, in \cite{vdV}, it is proven that 
there are no Abelian threefolds in $\mathbb{P}^6$.  
\end{rem}
Corollary \ref{ograniczenie na stopien} implies, in particular, that there is a finite number of families of
Calabi-Yau
threefolds in $\mathbb{P}^6$; three families of degenerate examples and a finite number of nondegenerate.
Indeed, by the Riemann--Roch theorem, $H.c_2(X)$ and the degree $H^3$ determine the Hilbert polynomial of a polarized Calabi--Yau threefold $(X,H)$. Moreover, if $(X,H)$ is a Calabi--Yau threefold in $\mathbb{P}^6$
polarized by its hyperplane section then, again by the Riemann--Roch theorem, $H.c_2(X)$
is determined by $H^3$ and $h^0(\mathcal{O}_X(H))=7$. It follows that the
Hilbert polynomial of $X$ is determined by the degree of $X$. Hence, in each degree, there is a
finite number of families. This means that having a bound on the degree implies
finiteness.

A slightly sharper bound on the degree could be obtained for Calabi--Yau
threefolds with $\rk \Pic(X)=h^{1,1}(X)=1$.
In this case, using the double point formula \ref{dpf}, we
obtain
\begin{equation} 2\geq 2(h^{1,1}(X)-h^{1,2}(X))=\chi_{top}(X) = -d^2+49d-588.\end{equation}
 We then infer that either $d\leq 21$ or $d\geq 28$. Let us, however, point out that there exist examples of Calabi--Yau threefolds $X\subset \mathbb{P}^6$ 
 with $h^{1,1}(X)>1$ (see \cite[cor 5.9.]{KK}).

\section{Quasi-Buchsbaum Calabi--Yau threefolds}\label{CY of decomposable bundles}

Recall the following definitions.

\begin{defi}\label{def ACM i QB}
 Let $X\subset \mathbb{P}^n$ be a subvariety in a projective space. Let us, for
each $i\in \mathbb{N}_{\geq 0}$, denote by $H^i_*(\mathcal{I}_X)$ the $i$-th cohomology module
$\bigoplus_{j\in \mathbb{Z}} H^i(\mathbb{P}^n,\mathcal{I}_X(j))$.
 We say that $X$ is arithmetically Cohen--Macaulay (aCM for short) if and only if
$H^i_*(\mathcal{I}_X)=0$ for $1\leq i \leq \dim(X)-1$. Moreover, $X$ is called
quasi-Buchsbaum if and only if, for $1\leq i \leq \dim(X)-1$, we have: $H^i_*(\mathcal{I}_X)$ is annihilated by the
maximal ideal of the structure ring of $\mathbb{P}^n$.
Finally, $X$ is arithmetically Buchsbaum if each of its linear sections is quasi-Buchsbaum.
\end{defi}

It is part of the mathematical folklore that the aCM Calabi--Yau threefolds in
$\mathbb{P}^6$ are only the ones listed in \cite{TO} up to degree 14. However, since we have not
found a proper proof of this fact in the literature, we provide it below as a
consequence of a more general result which will be important for the rest of the paper.
More precisely, we provide a classification of all quasi-Buchsbaum Calabi--Yau
threefolds in $\mathbb{P}^6$. In particular, this also gives a classification of all arithmetically Buchsbaum Calabi--Yau threefolds in $\mathbb{P}^6$.

\begin{thm} \label{Buchsbaum CY}
Let $X$ be a quasi-Buchsbaum Calabi--Yau threefold in $\mathbb{P}^6$.
Then $X=Pf(\varphi)$ for some $\varphi \in H^0(\mathbb{P}^6, \bigwedge^2 E(1))$
where $E$ is a vector bundle such that one of the following holds:
\begin{enumerate}
 \item $E=\bigoplus_{i=1}^{2u+1} \mathcal{O}_{\mathbb{P}^6}(a_i)$ with:
 \begin{enumerate}
  \item $u=1$, $a_1=-2$, $a_2=2$, $a_3=2$, and $X$ is a complete intersection of
type $1,1,5$;
  \item $u=1$, $a_1=-1$, $a_2=1$, $a_3=2$, and $X$ is a complete intersection of
type $1,2,4$;
  \item $u=1$, $a_1=0$, $a_2=0$, $a_3=2$, and $X$ is a complete intersection of
type $1,3,3$;
 \item $u=1$, $a_1=0$, $a_2=1$, $a_3=1$, and $X$ is a complete intersection of
type $2,2,3$;
 \item $u=2$, $a_i=0$, for $i\in\{1\dots 4\}$, $a_5=1$, and $X$ is a degree 13
Calabi--Yau threefold described in \cite{TO};
 \item $u=3$, $a_i=0$, for $i\in\{1\dots 7\}$, and $X$ is a degree 14 Calabi--Yau
threefold described in \cite{TO,Ro};
 \end{enumerate}
   \item  $E=\Omega^1_{\mathbb{P}^6}(1)\oplus
\bigoplus_{i=1}^{2v+1}\mathcal{O}_{\mathbb{P}^6}(a_i)$ with:
   \begin{enumerate}
 \item $v=0$, $a_1=1$, and $X$ is a degree 14 Calabi--Yau threefold from the
family  $\mathfrak{B}_{14}$ described in \cite{Be} (see also Example
\ref{przyklad stopnia 14 bertin});
 \item $v=1$, $a_1=0$, $a_2=0$, $a_3=0$, and $X$ is a degree 15 Calabi--Yau
threefold described in \cite{TO}.
\end{enumerate}
\end{enumerate}
\end{thm}
\begin{proof}
Take $X$ a quasi-Buchsbaum Calabi--Yau threefold in $\mathbb{P}^6$. Then, by definition, the
Hartshorne--Rao module $HR(X)=H^1_*(\mathcal{I}_X)$ is annihilated by the maximal ideal of the
structure ring of $\mathbb{P}^6$. 
It follows that the resolution of $HR(X)$ is given by a direct sum of twisted Koszul
complexes. Thus the sheafification $Syz^1(HR(X))$ of the first syzygy module of
$HR(X)$ is
$\bigoplus_{i=1}^{n_b} \Omega^1(b_i-2)$ for some $b_1\dots b_{n_b} \in \mathbb{Z}$.
We, moreover, claim that $b_i\leq 0$ for $i=1\dots n_b$.
 Indeed,  from Proposition \ref{pro linearly normal} and the exact sequence:
$$0\to \mathcal{I}_X(q) \to \mathcal{O}_{\mathbb{P}^6}(q) \to \mathcal{O}_{X}(q) \to 0.$$
we have $H^1(\mathcal{I}_X(q))=0$ for $q\leq 1$. To prove the claim, it is now enough to observe that in the above notation we have $H^1(\mathcal{I}_X(2-b_i))>0$ for $i=1\dots n_b$.

Let, now, $M$ be the module obtained by shifting the gradation in the Hartshorne--Rao module $HR(X)$ by 3. 
As observed in the introduction, there exists a vector bundle $E= Syz^1(M)\oplus
\bigoplus_{j=1}^{n_a} \mathcal{O}_{\mathbb{P}^6}(a_i)$ with $a_1\dots a_{n_a} \in
\mathbb{Z}$ and a section $\varphi \in H^0(\mathbb{P}^6, \bigwedge^2 E(1))$ such
that $X=Pf(\varphi)$.
Without loss of generality, we assume that $E$ is a bundle that has minimal rank among bundles for which there exists
such a $\varphi\in H^0(\mathbb{P}^6, \bigwedge^2 E(1))$ that $X=Pf(\varphi)$. The rank of the bundle $E$ is $2u+1$ for some $u\in \mathbb{N}$.
Observe that there is a decomposition
$$E= \bigoplus_{i=1}^{n_b} \Omega^1(b_i+1)\oplus \bigoplus_{j=1}^{n_a}
\mathcal{O}_{\mathbb{P}^6}(a_i).$$
We fix such a decomposition for the rest of the proof. The bundles appearing in this decomposition, treated both as subbundles and as quotient bundles, will be called components of $E$.
We also have an induced decomposition
$$E^*(-1)= \bigoplus_{i=1}^{n_b} (\Omega^1(b_i+1))^*(-1)\oplus \bigoplus_{j=1}^{n_a}
\mathcal{O}_{\mathbb{P}^6}(-a_i-1)=\bigoplus_{i=1}^{n_b} (\Omega^5(5-b_i))\oplus
\bigoplus_{j=1}^{n_a} \mathcal{O}_{\mathbb{P}^6}(-a_i-1).$$
One can, now, think of $\varphi$ as of a matrix consisting of blocks of the following
types
\begin{itemize}
 \item $\varphi^a_{i,j}=\pi^a_j \circ
\varphi|_{\mathcal{O}_{\mathbb{P}^6}(-a_i-1)}:
\mathcal{O}_{\mathbb{P}^6}(-a_i-1) \to \mathcal{O}_{\mathbb{P}^6}(a_j)$, for
$i,j \in \{1,\dots,n_a\}$;
 \item $\varphi^{a,b}_{i,j}= \pi^b_j \circ
\varphi|_{\mathcal{O}_{\mathbb{P}^6}(-a_i-1)}:
\mathcal{O}_{\mathbb{P}^6}(-a_i-1) \to  \Omega^1_{\mathbb{P}^6}(b_j+1) $, for
$i\in \{1,\dots,n_a\}$, $j\in \{1,\dots, n_b\}$ ;
 \item $\varphi^{a,b}_{i,j}= \pi^a_j \circ
\varphi|_{\Omega^5_{\mathbb{P}^6}(6-b_i)}: {\Omega^5_{\mathbb{P}^6}(5-b_i)} \to
\mathcal{O}_{\mathbb{P}^6}(a_j)$ for $i\in \{1,\dots,n_b\}$, $j\in \{1,\dots,
n_a\}$;
 \item $\varphi^b_{i,j}=\pi^b_j \circ \varphi|_{\Omega^5_{\mathbb{P}^6}(6-b_i)}
:{\Omega^5_{\mathbb{P}^6}(5-b_i)} \to  {\Omega^1_{\mathbb{P}^6}(b_j+1)}$ for
$i,j \in \{1,\dots,n_b\}$ ;
 \end{itemize}
 where $\pi^a_j$ and $\pi^b_j$ are the projections onto the components of $E$ i.e.
$\mathcal{O}_{\mathbb{P}^6}(a_j)$ and $\Omega^1_{\mathbb{P}^6}(b_j+1)$
respectively for each $j$.
We then have $\varphi^a_{i,j}=-(\varphi^a_{j,i})^*$,
$\varphi^{a,b}_{i,j}=-(\varphi^{b,a}_{j,i})^*$ and
$\varphi^b_{i,j}=-(\varphi^b_{j,i})^*$. Equivalently, we have a decomposition:
\begin{multline}
 H^0(\bigwedge^2 E(1))=\bigoplus_{1\leq i<j\leq n_a} H^0(\mathcal{O}_{\mathbb{P}^6}(a_i+a_j+1))\oplus \bigoplus_{\substack{
                                                                                                  i\in\{1\dots n_a\}\\ j\in\{1 \dots n_b\}
                                                                                                 }} H^0(\Omega^1_{\mathbb{P}^6}(a_i+b_j+2))\\
       \oplus \bigoplus_{1\leq i<j\leq n_b} H^0(\Omega^1_{\mathbb{P}^6} \otimes \Omega^1_{\mathbb{P}^6}(b_i+b_j+3)) \oplus \bigoplus_{i=1}^{n_b} H^0(\Omega^2_{\mathbb{P}^6}(2b_i+3))
\end{multline}

and the maps $\varphi^a _{i,j},\varphi^{a,b}_{i,j}, \varphi^b_{i,j}, \varphi^b_{i,i}$ are identified with the respective components of $\varphi$ in the above decomposition.
By abuse of notation,  we shall use this identification without further commentary throughout the proof.
We shall, moreover, use the notation $\varphi_{E_1,E_2}$ for the projection of $\varphi$ onto $(E_1\otimes E_2)(1)$ for any two subbundles $E_1$,$E_2 \subset E$ appearing as sums of components of $E$
and for the corresponding map $E_1^*(-1)\to E_2$.

Our aim is to deduce as many as possible restrictions on the coefficients $a_i$ and $b_i$
following from the fact that $\varphi$ defines a codimension 3 variety which has
trivial canonical class.
The first such condition is given by the formula \ref{pfaffian adjunction} for
the canonical class of a Pfaffian variety given in the introduction. We have
$$\omega_X=\mathcal{O}_X(2u+2c_1(E)-6),$$
which means that $c_1(E)+u=3$. Since $c_1(E)=\sum_{i=1}^{n_b} (-1+ 6b_i) + \sum_{j=1}^{n_a}
a_j$, we have the equation
\begin{equation} \label{adjunction for buchsbaum}
u+\sum_{i=1}^{n_b} (-1+ 6b_i) + \sum_{j=1}^{n_a} a_j =3.
\end{equation}
Since, in the above, both negative and positive $a_j$'s may occur, we, so far, still have infinitely many possibilities for the bundle $E$. To reduce this number of possibilities
to a finite number, we introduce the invariant $w(F)=c_1(F)+\frac{1}{2} \rk F\in \frac{1}{2}\mathbb{Z}$ for each subbundle $F$ of $E$. It is clear that $w$ is additive when taking direct sums of bundles. Moreover,
in terms of $w$, formula \ref{adjunction for buchsbaum} takes the form
\begin{equation} \label{adjunction for buchsbaum with w}
w(E)=3+\frac{1}{2}.
\end{equation}

Our aim is to find a natural decomposition of $E$ into a  direct sum of as many as possible bundles with positive $w$.  We proceed as follows.
Let $(l_i)_{i\in \{1\dots n_{\mathrm{neg}}\}}$ be the weakly increasing sequence consisting of all negative $a_i$'s and $b_i$'s with each
negative $b_i$ occurring 6 consecutive times (if no $a_i$ nor $b_i$ is negative we set $n_{neg}=0$). Similarly, let
$(k_i)_{i\in \{1\dots n_{\mathrm{pos}}\}}$ be the weakly decreasing sequence consisting of positive $a_i$'s (if no $a_i$ is positive we set $n_{neg}=0$).

The key to the proof of Theorem \ref{Buchsbaum CY} is the following lemma.
\begin{lemm}\label{lem first + last>0}
 Under the above notation, either $n_{\mathrm{neg}} = n_{\mathrm{pos}}=0$, or both the following hold:
 \begin{itemize}
 \item $n_{\mathrm{neg}} < n_{\mathrm{pos}}$,
  \item  $l_i+k_{i+1}\geq 0$ for each $i\in 1\dots n_{\mathrm{neg}}$.
 \end{itemize}
\end{lemm}
Before we pass to the proof of Lemma \ref{lem first + last>0}, let us finish the proof of Theorem \ref{Buchsbaum CY} assuming the lemma.

Keeping our notation, each component $F$ of $E$ such that $w(F)<0$ is either a line bundle  $\mathcal{O}_{\mathbb{P}^6}(l_{i(F)})$ for some
$i(F)\in \{1\dots n_{\mathrm{neg}}\}$ or a bundle of twisted first
differentials $\Omega^1_{\mathbb{P}^6}(l_{i(F)})$ with $l_{i(F)}=\dots=l_{i(F)+6}$ for some $i(F)\in \{1\dots n_{\mathrm{neg}}\}$.
Now, to each component $F$ of $E$ with $w(F)<0$ one associates a vector subbundle $A(F)$ of $E$ in the following way.

$$A(F):=\begin{cases}
         \mathcal{O}_{\mathbb{P}^6}(k_{i(F)+1}) \text{ for } F=\mathcal{O}_{\mathbb{P}^6}(l_{i(F)})\\
         \bigoplus_{j=1}^6 \mathcal{O}_{\mathbb{P}^6}(k_{i(F)+j}) \text{ for } F=\Omega^1_{\mathbb{P}^6}(l_{i(F)})
        \end{cases}$$

By Lemma \ref{lem first + last>0}, the bundle $A(F)$ is well defined for every component $F$ of $E$ with $w(F)<0$ and the following holds:
\begin{itemize}
\item $A(F)$ is a sum of components of positive degree not involving the component $\mathcal{O}_{\mathbb{P}^6}(k_{1})$;
\item if $F_1\cap F_2=0$ then $A(F_1)\cap A(F_2)=0$;
\item $\rk A(F)=\rk F$;
\item $w(F)+w(A(F))\geq \rk F$.
\end{itemize}
It follows that we get a decomposition of $E$:
$$E=\bigoplus_{\substack{F \text{ component of } E\\ w(F)<0}} (F\oplus A(F)) \oplus \bigoplus_{\substack{F \text{ component of } E\\w(F)> 0\\
F\cap \bigoplus_{w(G)<0} A(G)=0
}} F$$

into bundles with positive $w$ (observe that $w(F)\neq 0$ when $F$ is a component of $E$) including the line bundle of maximal degree $\mathcal{O}_{\mathbb{P}^6}(k_{1})$. We then easily list
all possibilities. Indeed, we have:
\begin{enumerate}
 \item $b_i= 0$ for $i\in\{1\dots n_b\}$, because for any $c\leq 0$, we have $w(\Omega^1_{\mathbb{P}^6}(c)+A(\Omega^1_{\mathbb{P}^6}(c)))\geq 6>3$ ;
\item $n_b\leq 1$, because $w(\Omega^1_{\mathbb{P}^6}(1))=2$;
\item if $n_b=1$ then there is no negative $a_i$, because otherwise $w(\mathcal{O}_{\mathbb{P}^6}(l_{1})\oplus \mathcal{O}_{\mathbb{P}^6}(k_{1}) \oplus \mathcal{O}_{\mathbb{P}^6}(k_{2}))\geq 2+\frac{1}{2}$;
\item if $n_b=0$ then $l_1 \geq -2$, because when $l_1\leq -3$, we have  $w(\mathcal{O}_{\mathbb{P}^6}(l_{1})\oplus \mathcal{O}_{\mathbb{P}^6}(k_{1}) \oplus \mathcal{O}_{\mathbb{P}^6}(k_{2})\geq 4+ \frac{1}{2}$;
\item if $l_1 = -2$ then $E=\mathcal{O}_{\mathbb{P}^6}(-2)\oplus \mathcal{O}_{\mathbb{P}^6}(2)$, because
$w(\mathcal{O}_{\mathbb{P}^6}(l_{1})\oplus \mathcal{O}_{\mathbb{P}^6}(k_{1}) \oplus \mathcal{O}_{\mathbb{P}^6}(k_{2})\geq 3+\frac{1}{2}$ and equality holds only if $k_1=k_2=2$;
\item if $l_1 = l_2=-1$ then  $E=3\mathcal{O}_{\mathbb{P}^6}(1)\oplus 2\mathcal{O}_{\mathbb{P}^6}(-1)$, because $w(\mathcal{O}_{\mathbb{P}^6}(l_{1})\oplus \mathcal{O}_{\mathbb{P}^6}(l_{2})\oplus
\mathcal{O}_{\mathbb{P}^6}(k_{1}) \oplus \mathcal{O}_{\mathbb{P}^6}(k_{2}) \oplus \mathcal{O}_{\mathbb{P}^6}(k_{3})\geq 3+\frac{1}{2}$ and equality holds only if $k_1=k_2=k_3=1$;
\item if $l_1 = -1$ and all remaining $a_i$ are nonnegative then we have two possibilities:
\begin{enumerate}
\item $E=\mathcal{O}_{\mathbb{P}^6}(-1)\oplus \mathcal{O}_{\mathbb{P}^6}(1) \oplus \mathcal{O}_{\mathbb{P}^6}(2)$
\item $E=\mathcal{O}_{\mathbb{P}^6}(-1)\oplus 2 \mathcal{O}_{\mathbb{P}^6}(1) \oplus 2\mathcal{O}_{\mathbb{P}^6}$ and the constant terms in $\varphi$ are $0$
\end{enumerate}
\item if $n_b=0$ and all $a_i$ are non-negative we have 4 possibilities for $E$ as in the assertion.
\end{enumerate}

To conclude, we need to exclude two cases which do not appear in the assertion:
\begin{itemize}
\item $E=3\mathcal{O}_{\mathbb{P}^6}(1)\oplus 2\mathcal{O}_{\mathbb{P}^6}(-1)$; we shall see in Proposition \ref{CY stopnia 11} that Pfaffian varieties associated to this bundle are always singular.
\item $E=\mathcal{O}_{\mathbb{P}^6}(-1)\oplus 2 \mathcal{O}_{\mathbb{P}^6}(1) \oplus 2\mathcal{O}_{\mathbb{P}^6}$ and the constant terms in $\varphi\in \bigwedge^2 E(1)$ are $0$; we easily see that
$Pf(\varphi)$ either does not exist (i.e. the degeneracy locus is of codimension $\leq 2$) or must be contained in a hyperplane.
\end{itemize}

To complete the proof of Theorem \ref{Buchsbaum CY} we, hence, need only to prove Lemma \ref{lem first + last>0} and Proposition \ref{CY stopnia 11}.
\end{proof}

\begin{proof}[Proof of Lemma \ref{lem first + last>0}] If $n_{neg}=0$ then the assertion is trivial.
Assume, hence, by contradiction that $n_{neq} >0$ and that there exists $i\in \{1\dots n_{\mathrm{neg}}\}$ such that $i+1>n_{pos}$ or that $k_{i+1}+l_i < 0$. Let 
$$E_1=\bigoplus_{\{j| a_j\leq l_i\}} \mathcal{O}_{\mathbb{P}^6}(a_j) \oplus \bigoplus_{\{j| b_j\leq l_i\}} \Omega^1_{\mathbb{P}^6}(1+b_j)$$ 
considered as a subbundle of $E$. Observe that, by definition of the sequence $(l_i)_{i\in \{1\dots n_{\mathrm{neg}}\}}$, we have $\rk E_1\geq i$.
Let, moreover, 
$$E_2= \begin{cases} \bigoplus_{\{j| a_j\leq k_{i+1}\}} \mathcal{O}_{\mathbb{P}^6}(a_j) \oplus \bigoplus_{\{j| b_j\leq k_{i+1}\}} \Omega^1_{\mathbb{P}^6}(1+b_j) \text{ when } i+1\leq n_{pos}\\

\bigoplus_{\{j| a_j\leq 0\}} \mathcal{O}_{\mathbb{P}^6}(a_j) \oplus \bigoplus_{\{j| b_j\leq 0\}} \Omega^1_{\mathbb{P}^6}(1+b_j) \text{ when } i+1> n_{pos}

       \end{cases}
$$
also considered as a subbundle of $E$. Similarly as for $E_1$ we have $\rk E_2\geq 2u+1-n_{pos}$ when  $i+1> n_{pos}$ and $\rk E_2\geq 2u+1-i$.
In any case:
\begin{equation}\label{rkE1+rkE2>E}
\rk E_1 +\rk E_2 \geq \rk E=2u+1.
\end{equation}
We also clearly have $E_1\subset E_2$.
Moreover, the following lemma proves that  $\varphi_{E_1,E_2}=0$.

\begin{lemm}\label{lemma no nonzero constant terms in Pfaffian}
In the notation of the proof of Theorem \ref{Buchsbaum CY}, we have:
\begin{itemize}
 \item $\varphi^a_{i,j}=0$ when $a_i+a_j\leq -1$
 \item $\varphi^{a,b}_{i,j}=0$ and $\varphi^{b,a}_{j,i}=0$ when $a_i+b_j\leq -1$
 \item $\varphi^{b}_{i,j}=0$ when $b_i+b_j\leq 0$
\end{itemize}
\end{lemm}
\begin{proof}
Using the following vanishing
\begin{itemize}
 \item $Hom(\mathcal{O}_{\mathbb{P}^6}(-a-1), \mathcal{O}_{\mathbb{P}^6}(b))= H^0(\mathcal{O}_{\mathbb{P}^6}(b+a+1))=0$
when $a+b\leq -2$
 \item $Hom(\mathcal{O}_{\mathbb{P}^6}(-a-1), \Omega^1_{\mathbb{P}^6}(b+1))=H^0(\Omega^1_{\mathbb{P}^6}(a+b+2))=0$
when $a+b\leq -1$
 \item $Hom((\Omega^1_{\mathbb{P}^6})^*(-a-2), \Omega^1_{\mathbb{P}^6}(b+1))=H^0(\Omega^2_{\mathbb{P}^6}(a+b+3))=0$
when $a+b\leq -1$
\end{itemize}
we get
\begin{itemize}
 \item $\varphi^a_{i,j}=0$ when $a_i+a_j\leq -2$
 \item $\varphi^{a,b}_{i,j}=0$ and $\varphi^{b,a}_{j,i}=0$ when $a_i+b_j\leq -1$
 \item $\varphi^{b}_{i,j}=0$ when $b_i+b_j\leq 0$
\end{itemize}

It remains to prove that $\varphi^a_{i,j}=0$ when $a_i+a_j=-1$.
  Assume the contrary, then $\varphi^a_{i,j}$, for some $i,j \in \{1\dots l\}$, is a non-zero section of $\mathcal{O}_{\mathbb{P}^6}$ thus a non-zero constant which we assume to be 1 by re-scaling.
  Denote by $E'$ the subbundle of $E$ such that $E=E'\oplus \mathcal{O}_{\mathbb{P}^6}(a_i)\oplus \mathcal{O}_{\mathbb{P}^6}(a_j)$. The section $\varphi$ is then decomposed as a sum
  $$\varphi=(1,\varphi_{\mathcal{O}_{\mathbb{P}^6}(a_i),E'}, \varphi_{\mathcal{O}_{\mathbb{P}^6}(a_j),E'}, \varphi_{E',E'})\in H^0(\mathcal{O}_{\mathbb{P}^6} \oplus E'(a_i+1)\oplus E'(a_j+1)\oplus \bigwedge^2 E'(1))$$
  Consider $$\varphi'=(1,0,0,\varphi_{E',E'} + (\varphi_{\mathcal{O}_{\mathbb{P}^6}(a_i),E'}\wedge \varphi_{\mathcal{O}_{\mathbb{P}^6}(a_j),E'})).$$ We claim that $\Pf(\varphi)=\Pf(\varphi')=\Pf(\psi)$, where
  $\psi=\varphi_{E',E'} + (\varphi_{\mathcal{O}_{\mathbb{P}^6}(a_i),E'}\wedge \varphi_{\mathcal{O}_{\mathbb{P}^6}(a_j),E'})\in H^0(\bigwedge^2 E'(1))$. Indeed, $\Pf(\varphi)=\Pf(\varphi')$ follows from the fact that 
  locally under a trivialization of $E$ respecting the decomposition $E=E'\oplus \mathcal{O}_{\mathbb{P}^6}(a_i)\oplus \mathcal{O}_{\mathbb{P}^6}(a_j)$ we have $\varphi'$ is obtained by row and column operations from $\varphi$.
  On the other hand, the equality $\Pf(\varphi')=\Pf(\psi)$ is clear.
  The claim being proven, we get a contradiction with the minimality of $E$. This shows that $\varphi^a_{i,j}=0$ when $a_i+a_j=-1$ and completes the proof.
 \end{proof}

We conclude the proof of Lemma \ref{lem first + last>0} by obtaining a contradiction of the above with the following lemma.

\begin{lemm} \label{lem too big blocks of zeroes} Let $E_1\subset E_2 \subset E$ be subbundles of $E$ given by
some sums of its components. Consider $\varphi_{E_1,E_2}: E_1^*(-1)\to
E_2$ as defined above. If $\varphi_{E_1,E_2}=0$ then
$\rk(E_1)+\rk(E_2)<2u+1$.
 \end{lemm}
\begin{proof}
Consider the map $\varphi_{E_1,E_2^c}$, where \emph{$E_2^c$ is the
subbundle of $E$} being the sum of those components of $E$ which are not contained
in $E_2$.
Under our assumptions, we clearly have $D_{2u-1} (\varphi)\supset D_{(\rk(E_1)-2)}(\varphi_{E_1,E_2^c})$.
Now, since $D_{2u-1} (\varphi)$ is not the whole space, we have $2u+1-\rk E_2= \rk
E_2^c\geq \rk E_1 -1$ thus $\rk(E_1)+\rk(E_2)\leq 2u+2$.

To exclude $\rk(E_1)+\rk(E_2)=2u+2$, we observe that in such case we have
$D_{(\rk(E_1)-2)}(\varphi_{E_1,E_2^c})$ is either empty or of codimension at
most 2. It cannot be of codimension at most 2 as it is contained in a Pfaffian variety, hence, it must be empty. If it is empty then $ \varphi_{E_1,E_2^c}$
induces an embedding of vector bundles $E_1^*(-1)\to E_2^c$. This means that we have an exact sequence
$$0\to E_1^*(-1)\to E_2^c \to L\to 0,$$
where $L$ is a line bundle on $\mathbb{P}^6$. Now, since $E_1$ is a direct sum of line bundles  and of twisted first differentials on $\mathbb{P}^6$, we have $\Ext^1_{\mathbb{P}^6}(L,E_1^*(-1))=0$.
We thus have  $E_2^c\simeq E_1^*(-1)\oplus L$. It follows that $ \varphi_{E_1,E_2^c}$ consists of blocks of the form
$\varphi^a_{i,j}$ for some $i,j\in \mathbb{Z}$ with constant entries which by minimality of $E$ and Lemma \ref{lemma no nonzero constant terms in Pfaffian} are zero. 
This leads to a contradiction with $\varphi_{E_1,E_2^c}$ being an embedding and proves that
$\rk(E_1)+\rk(E_2)\neq 2u+2$.

To exclude $\rk(E_1)+\rk(E_2)=2u+1$, we shall prove that, in this case, we have  $Y=\Pf(\varphi)\cap
D_{(\rk(E_1)-1)}(\varphi_{E_1,E_2^c})$ is contained in $\Pf(\varphi)$ and is of codimension at most 2 in
$\mathbb{P}^6$.
To see the latter, we describe  $Y$ as a degeneracy locus of a map between
vector bundles of expected codimension 2. More precisely, we claim that $Y=D_{(\rk(E_1)-1)}
(\kappa),$ where
$\kappa=(\kappa_1,\kappa_2): E_1^*(-1) \to (\det E_3)(u-\rk E_1) \oplus E_2^c$ with $\kappa_2=
\varphi_{E_1,E_2^c}$ and $$\kappa_1=\varphi_{E_1,E_3}\wedge
(\varphi_{E_3,E_3})^{\wedge(u-\rk(E_1))} \in H^0((E_1\otimes \det(E_3))(u+1-\rk E_1)).$$ 
Indeed, the claim follows directly from the following observation:

\begin{multline}\varphi^{\wedge u}=(\kappa_1 \wedge (\varphi_{E_1,E_2^c})^{\wedge(\rk(E_1)-1)}, (\varphi_{E_1,E_2^c})^{\wedge \rk E_1} \wedge \gamma)
\\ \in H^0((\det E_3\otimes \det(E_1)\otimes \bigwedge^{rk E_1-1} E_2^c) (u))\oplus H^0( (\det(E_1)\otimes \bigwedge^{\rk E_3 -1} E_3 \otimes \det E_2^c)(u))
\subset H^0(\bigwedge^{2u} E (u))
\end{multline}
for some $\gamma \in \bigwedge^{\rk E_3 -1} E_3$.

Now, if $D_{(\rk(E_1)-1)}(\varphi_{E_1,E_2^c})$ is non-empty then it is
a hypersurface in $\mathbb{P}^6$ and in consequence $Y=\Pf(\varphi)\cap
D_{(\rk(E_1)-1)}(\varphi_{E_1,E_2^c})$ is also nonempty. The codimension of $Y$ is then at most $2$ giving a contradiction with $Y\subset \Pf(\varphi)$.
If $D_{(\rk(E_1)-1)}(\varphi_{E_1,E_2^c})$ is empty we have $E_2^c\simeq E_1^*(-1)$ and we get a contradiction
with the minimality of $E$ as in the case $\rk(E_1)+\rk(E_2)=2u+2$.
%
\end{proof}
Inequality \ref{rkE1+rkE2>E}, holding under the assumption that there exists an $i\in \{1\dots n_{\mathrm{neg}}\}$ such that $i+1>n_{pos}$ or that $k_{i+1}+l_i < 0$, together with Lemma  
\ref{lemma no nonzero constant terms in Pfaffian} gives a contradiction with Lemma \ref{lem too big blocks of zeroes} and thus finishes the proof of Lemma \ref{lem first + last>0}.
\end{proof}

The following is a straightforward Corollary from Theorem \ref{Buchsbaum CY}.
\begin{cor}\label{Cor from Buchsbaum CY}
All quasi Buchsbaum Calabi--Yau threefolds in $\mathbb{P}^6$ are arithmetically Buchsbaum. Moreover, arithmetically Cohen Macaulay Calabi--Yau threefolds in $\mathbb{P}^6$ are exactly those which correspond to case (1)
in Theorem \ref{Buchsbaum CY}.
\end{cor}

\section{The degrees of Calabi--Yau threefolds in quadrics} \label{sec deg CY in quadrics}
The approach to the classification of Calabi--Yau threefold in $\PP^6$ by considering the minimal degrees of hypersurfaces that contain the Calabi--Yau threefolds is inspired by \cite{H}.
In this section, we shall work on quadrics of different ranks and dimensions. We
shall use the following notation.
For $k\leq 5$, we shall denote by $Q^r_k\subset \mathbb{P}^{k+1}$ a
$k$-dimensional quadric
hypersurface with corank $r$.

\begin{prop} \label{gen typu w gladkiej kwadryce} Let $Q^0_4 \subset
\mathbb{P}^5$ be a smooth 4-dimensional quadric hypersurface. If $S$ is a
nondegenerate canonically
embedded surface of general type of degree $d$ contained in $Q^0_4$,
then  $12\leq d\leq 14$.
\end{prop}
\begin{proof} Take $S$ a canonical surface of degree $d$ contained in a smooth 4-dimensional quadric  $Q^0_4$.
By the analog of the double point formula for quadrics (see \cite[thm.~9.3]{Ful}), the
second Chern class of the normal bundle $c_2(N_{S|Q^0_4})= S\cdot S$.
Now, the Chow group $A^2(Q^0_4)$ of $Q^0_4$ has two generators corresponding to two
families of
planes on $Q^0_4$. We shall denote them $\theta_1$ and $\theta_2$. We have
$\theta_1^2=\theta_2^2=1$ and $\theta_1\cdot \theta_2=0$.
 By definition we know that $S\cdot (\theta_1+\theta_2)=d$, hence
 $S\sim a\theta_1+(d-a)\theta_2$ for some $a\in \mathbb{Z}$.
It follows that $$c_2(N_{S|Q^0_4})=2a^2+d^2-2da.$$

On the other hand, from the exact sequence
$$0\to T_{S}\to T_{{Q^0_4}}|_S\to N_{S|{Q^0_4}}\to 0,$$
we find that $c_1(N_{S|{Q^0_4}})=5h$ and
$$c_2(N_{S|{Q^0_4}})=12h^2-c_2(S)=12d-c_2(S),$$
 where $h=K_S$ is the restriction of the hyperplane class from $\mathbb{P}^5$ to
$S$. It follows that $$c_2(S)=12d-2a^2-d^2+2ad.$$
On the other hand, using the Riemann-Roch formula we have\\
$7=\frac{1}{12}(d+c_2(S))$, thus $c_2(S)=84-d$. By comparing both formulas, we
infer
\begin{equation} \label{rownanie kwadratowe na d dla gladkiej kwadryki}
84-d=12d-2a^2-d^2-2ad.
\end{equation}
The only integers $d$ for which this equation has a solution are $d=12,13,14$.
\end{proof}

Let us now consider a similar problem for a singular quadric.
\begin{prop}\label{prop1} The degree $d$ of a nondegenerate canonically embedded surface of
general type $S\subset\mathbb{P}^5$ contained in a quadric cone $Q_4^1\subset \mathbb{P}^5$ is either
$12$ or $13$ or $14$.
\end{prop}
\begin{proof} Take $S$ a canonical surface of degree $d$ contained in a 4-dimensional quadric  $Q^1_4$ of corank 1.
Let us consider
the projective bundle $$g\colon
\mathfrak{F}:=\mathbb{P}(\mathcal{O}_{Q^0_3}\oplus \mathcal{O}_{Q_3^0}(1))\to Q_3^0,$$
of the 1-dimensional quotients of the vector bundle
$F:=\mathcal{O}_{Q_3^0}\oplus \mathcal{O}_{Q_3^0}(1).$
The linear system of its tautological line bundle $\mathcal{O}_{\mathfrak{F}}(1)$
defines a morphism $\pi\colon \mathfrak{F}\to Q_4^1\subset \mathbb{P}^5$. Observe that $\pi$ is
the blow up of $Q_4^1$ in the singular point.
 Denote by $\xi$ and $h$ the pull back to $\mathfrak{F}$ by $\pi$ and $g$
respectively of the hyperplane sections of $Q_4^1$ and $Q_3^0$.
Let $S'$ be the proper transform of $S\subset Q_4^1$ on $\mathfrak{F}$. By the
Grothendieck relation, we have $\xi^2=h\xi$ so that $A^2(\mathfrak{F})$ is
generated by $\xi^2$ and $h^2$. We can, hence, write the class of $S'$ as
$[S']=a\xi^2+
b h^2$ where $a,b\in \mathbb{Q}$.

Observe that $S'$ is smooth and either isomorphic to $S$ or the blow up of $S$
in the center of the cone $Q_4^1$ (when $S$ contains this center).
From the double point formula, we infer $$c_2(N_{S'|\mathfrak{F}})=(a\xi^2+ b
h^2)^2=\xi^4(a^2+2ab)=2a(d-a),$$ since $d=[S']\xi^2=2a+2b.$
Now, let us compute this Chern class using the exact sequence   $$0\to T_{S'}\to
T_{\mathfrak{F}}|_{S'}\to N_{S'|\mathfrak{F}}\to 0.$$
We deduce that
$$c_2(N_{S'|\mathfrak{F}})=c_2(\mathfrak{F})-c_2(S')-c_1(S')c_1(N_{S'|\mathfrak{
P}}).$$
Next, we use, $$0\to \mathcal{O}_{\mathfrak{F}}\to g^{\ast}F^{\ast}(1)\to
T_{\mathfrak{F}|Q_3^0} \to 0 $$
$$0\to T_{\mathfrak{F}|Q_3^0} \to T_{\mathfrak{F}}\to g^{\ast} T_{Q_3^0} \to 0
.$$
We obtain $c_2(N_{S'|\mathfrak{F}})=12d-84+2a$ using the fact that
$c_1(S')=(-2\xi+h)|_{S'}$ and the Noether formula $c_2(S')=84-h^2[S']=84-2a$.
We infer the equation: $$a^2+a(1-d)+6d-42=(a-6)(a-d+7)=0.$$
We thus deduce that either $a=6$ and $b=\frac{d}{2}-6$ or $a=d-7$ and
$b=7-\frac{1}{2}d.$

Now, the exceptional divisor $\Xi\subset \mathfrak{F}$ of $\pi$ is in the linear
system $|\xi-h|$.
Since $S$ is smooth, we have one of the following:
\begin{itemize}
 \item either $\Xi|_{S'}=0$ and $S'$ is isomorphic to $S$;
 \item or  $(\Xi|_{S'})^2=-1$ and $S'$ is the blow up of $S$ in one point.
\end{itemize}

In the first case, we deduce that $b=0$ so either $d=12$ or $d=14$.
In the second case, $b=\frac{1}{2}$ so $d=13$.
\end{proof}
\begin{thm}\label{tw kw} The degree $d$ of a nondegenerate Calabi--Yau threefold $X$ embedded
in a quadric in $\mathbb{P}^6$ of corank $\leq 2$ is bounded by $12\leq d\leq 14$.

Moreover, the degree of a nondegenerate Calabi-Yau threefold contained in a smooth quadric
cannot be 13.
\end{thm}
\begin{proof} First, remark that a generic hyperplane section of $X\subset
\mathbb{P}^6$ is a canonically embedded surface of general type contained 
either in a smooth quadric $Q_4^0$ or in a corank 1 quadric $Q_4^1 \subset \mathbb{P}^5$. We obtain the bound on the degree from Propositions
\ref{prop1} and \ref{gen typu w gladkiej kwadryce}.

To prove the second part, observe that  $A^2(Q_5^0)$ has only one generator $\theta$
and that the restriction $\theta|_H$ to a hyperplane section $H=Q^0_4$ is $\theta_1+\theta_2$ in the notation of Proposition \ref{prop1}. It
follows that, in this case, in
the proof of Proposition \ref{gen typu w gladkiej kwadryce}, we may additionally
assume $a=b$.
It is now enough to observe that the equation (\ref{rownanie kwadratowe na d dla
gladkiej kwadryki}) has no solution for $d=13$ giving $a=b$. In fact,
when $d=13$, we must have $S\sim 6\theta_1+7\theta_2$.
\end{proof}
To conclude the proof of Theorem \ref{Main theorem}, it is enough to prove the
following:
\begin{prop}\label{no CY in quadrics of rank 4}
If $X$ is a smooth Calabi--Yau threefolds contained in a quadric $Q_5^r$ of corank $r\geq
3$ in $\mathbb{P}^6$ then $X$ is contained in a linear space contained in $Q_5^r$.
\end{prop}
\begin{proof}
The proof will consist of two steps which will be formulated as lemmas.
\begin{lemm} Let $X$ be a smooth threefold contained in a quadric $Q_5^r$ of corank $r\geq
3$ in
$\mathbb{P}^6$. Then one of the following holds:
\begin{enumerate}
\item $X$ is contained in a linear space $L\subset Q_5^r$
\item $X$ is fibered by surfaces contained in
linear spaces contained in $Q^r_5$.
\end{enumerate}
\end{lemm}
\begin{proof}
If the quadric $Q_5^r$ is of corank $\geq 5$ then it decomposes into hypersurface
and $X$ being smooth must be contained in one of them.

If the quadric $Q_5^r$ has corank 4 then it is a cone $Q_5^4$ with center some
$\mathbf{P}\simeq \mathbb{P}^3$ and spanned over a smooth conic $Q_1^0$. Consider the hyperplane $H_l$
spanned by
$\mathbf{P}$ and any line in the plane spanned by $Q_1^0$. Observe that $H_l\cap Q$
decompose as a sum of two linear spaces $L_l^1$ and $L_l^2$.
It follows that $H_l\cap X$ is either the whole $X$ for some $l$ or
is a divisor on $X$ decomposed as $S_L=S_{l}^{1}+S_l^2$ with $S_l^1\subset
L^1_{l}$ and $S_l^2\subset L_l^2$.
In the first case the assertion follows as $H_l \cap Q_5^4$ is a union of two linear
spaces and $X$ is irreducible.
We, hence, need only to consider $X$ having a 1-dimensional linear system given
by surfaces $S_l^i$ which are clearly all in the same system independently on
$i\in \{1,2\}$.
Take two distinct points $q,q'\in Q_1^0$. For a 1-dimensional linear system on a smooth threefold,  we
have two possibilities:
\begin{itemize}
 \item the system is base point free, in which case we have a fibration of $X$
given by surfaces $S_q$ giving the assertion,
\item or the system has a base curve $C$.
\end{itemize}
 In the latter case,
we claim that $\mathbf{P}$ must be the space tangent to $X$ in every point of
this curve. Indeed, all hyperplanes containing $\mathbf{P}$ intersect $X$ in the
union of
two surfaces passing through $C$, hence, the intersection is singular in every
point of $C$. It follows that the tangent space to $X$ in any of the points of $C$ is contained
in $\mathbf{P}$ thus is equal to $\mathbf{P}$. We now use the famous Zak
theorem on
tangencies, which for smooth varieties is formulated as follows.
\begin{thm}[Zak \cite{Z}]
Let $X^n\subset \mathbb{P}^{n+a}$ be a subvariety not contained in a hyperplane.
Fix any linear space $L=\mathbb{P}^{n+k}\subset \mathbb{P}^{n+a},
0\leq k\leq a-1$. Then
$$\dim\{x\in X|\tilde{T}_x X\subset L\} \leq k,$$
where $\tilde{T}_x X$ denotes the embedded tangent space to $X$ in $x$.
\end{thm}
Zak's theorem implies that a tangent projective space to a threefold in $\mathbb{P}^6$
cannot be tangent to it in a curve. In particular $\mathbf{P}$ cannot be tangent to $X$ in $C$. The contradiction concludes the proof for $r=4$. 
Let us point out that from now on we shall consider $\tilde{T}_x X$ as standard notation for the embedded tangent space.

Finally, assume that $r=3$ i.e. the quadric $Q_5^r$ is of rank $4$. Then $Q_5^3$ is a cone with
center a plane $\mathbf{P}$ and spanned over a smooth quadric $Q_2^0$ of dimension $2$.
Consider the family $\{X_q\}_{q\in Q_2^0}$ consisting of intersections of $X$ with hyperplanes $H_q$
spanned by $\mathbf{P}$ and planes $\tilde{T}_q Q_2^0$ tangent to $Q_2^0$
parametrized by the tangency points. Observe that $\tilde{T}_qQ_2^0\cap Q_2^0$ consists of
two lines $l^q_1$ and $l_2^q$ with $q=l^q_1 \cap l_2^q$,
hence, $H_q\cap Q_5^3$ is the union of two 4-dimensional linear spaces $L_1^q$ and
$L_2^q$ spanned by $\mathbf{P}$ and lines $l^q_1$ and $l_2^q$
respectively.
If $X_q=X$ for some $q\in Q_2^0$ then $X$, being irreducible, is contained in one of
the two linear spaces $L_i^p$ for $i=1$ or $2$
and the assertion follows.
If $X_q\neq X$ for all $q\in Q_2^0$ then $X_q$ is decomposed as $X_q=S^q_1\cup
S^q_2$.
We thus have two 1-dimensional linear systems on $X$ and as in the previous
case we have the following possibilities
\begin{itemize}
 \item one of the linear systems is base point free giving us a fibration by
surfaces as in the assertion
\item both systems contain a base curve.
\end{itemize}
In the latter case, let us denote the base curves by $C_1$ and $C_2$. Then
$C_1\cap C_2\neq \emptyset$ as these are plane curves. Let us choose any point
$p\in C_1\cap C_2$.
Now, for every $q$, we have  $H_q\cap X=S^q_1\cup S^q_2$ is a reducible surface which is
clearly singular in $p\in C_1\cap C_2\subset S^q_1\cap S^q_2$.
It follows that $\tilde{T}_pX \subset H_q$ for all $q\in Q_2^0$. Thus
$\tilde{T}_pX\subset \bigcap_{q\in Q^3_5} H_q= \mathbf{P}$. As $\tilde{T}_pX$ is of
dimension 3 and
$\mathbf{P}$ is a plane, we get a contradiction excluding the last case and finishing the proof.
\end{proof}
\begin{lemm}\label{no fibred CY in quadrics}
 There are no smooth Calabi--Yau threefolds  contained in a quadric $Q_5^{r}$ of corank
$3\leq r\leq 4$ which admit a fibration by surfaces contained in linear spaces
contained in $Q^r_5$.
\end{lemm}
\begin{proof} Assume the contrary. Let $X$ be a smooth Calabi--Yau threefold
contained in $Q^r_5$ such that $X$ admits a fibration $f:X\to Y$ with fibers being
surfaces contained in
linear spaces contained in $Q^r_5$. As usual we use the notation for which $Q^r_5$ will
be a cone with center a linear space $\mathbf{P}$ and spanned over a smooth
quadric $Q_{5-r}^0$ of dimension $5-r$. As $X$ is a Calabi--Yau threefold,
we have $Y\simeq \mathbb{P}^1$ and the fiber $X_y$ for a generic point $y\in Y$
is either a
K3 or an abelian surface (smooth in both cases by Bertini theorem). Observe
moreover that the maximal dimensional linear spaces in $Q^r_5$ are projective spaces
of dimension 4. In both cases, we have only 1-dimensional
families of $\mathbb{P}^4$ contained in $Q^r_5$ and no two fibers of the fibration
can be contained in the same $\mathbb{P}^4$ as otherwise they would have to
intersect.
It follows that we can choose a family $ \{L_y\}_{y\in Y}$ such that
$L_y\simeq\mathbb{P}^4\subset Q^r_5$ and $L_y\cap L_y'=\mathbf{P}$ for any distinct
$y,y'\in Y$ such that
for a generic $y\in Y$ the fiber $X_y$ is either a K3 or an abelian surface
contained in $L_y$.

\begin{Claim} We claim that such a fiber $X_y$ must be one of the following:
\begin{itemize}
 \item a complete intersection of a quadric and a cubic in $L_y\simeq \mathbb{P}^4$
\item a complete intersection of a hyperplane and a quartic in
$L_y\simeq \mathbb{P}^4$
\item the zero locus of the Horrocks Mumford bundle on $L_y\simeq \mathbb{P}^4$
\end{itemize}
\end{Claim}
Indeed, fix a generic $y\in Y$. The fiber $X_y$ is then a smooth K3 or abelian
surface contained in $L_y=\mathbb{P}^4$.
The theorem of Severi, saying that the only surface in
$\mathbb{P}^4$ which is not linearly normal is the Veronese surface, implies
that $X_y$ is linearly normal. Let us denote its degree by $d_{X_y}$. The following
computation applies:

By the double point formula, $c_2(N_{X_y|\mathbb{P}^4})=[X_y]^2=d_{X_y}^2$.
On the other hand, by the exact sequence:
$$0\rightarrow T_{X_y}\rightarrow T_{\mathbb{P}^4}|_S\rightarrow
N_{X_y|\mathbb{P}^4}\rightarrow 0,$$
and the vanishing $c_1(T_{X_y})=0$, we have:
$$c_2(N_{X_y|\mathbb{P}^4})=c_2(T_{\mathbb{P}^4}|_{X_y})-c_2(T_{X_y})=\begin{cases}
                                                             10d_{X_y}-24 \quad
\text{when $X_y$ is a K3 surface} \\ 10d_{X_y} \quad \text{when $X_y$ is an abelian
surface}
                                                            \end{cases}
$$
Comparing the two equations, we have $d_{X_y}=4$ or $6$ and $X_y$ is a K3 surface or
$d_{X_y}=10$ and $X_y$ is an abelian surface. Finally, we use the classification
(cf. \cite[table 7.3]{DP}) of non-general type surfaces in $\mathbb{P}^4$ with $d_{X_y}=4,6,10$ to obtain our
claim.

From now on, we shall consider the cases of rank $r=3$ and $r=4$ separately.
For $r=3$, we have $\mathbf{P}$ is a 3-dimensional space and $Q^4_5$ is a cone
centered at $\mathbf{P}$ and spanned over a conic $Q^0_1$. A generic fiber $X_y$
meets $\mathbf{P}$ in a curve $C_y$ and all these curves must be disjoint.
It follows that $X\cap \mathbf{P}$ contains a surface $\mathbf{S}$ fibered by the
curves $C_y$. Observe, moreover, that since the fibers $X_y$ cover $X$ it follows
that  $\mathbf{S}=X\cap \mathbf{P}$ is of pure dimension 2. Now,
by the Zak tangencies theorem $\mathbf{S}$ is smooth in codimension 1 (since any
singular point of $\mathbf{S}$ is a tangency point of $\mathbf{P}$ with $X$) and
since it is a hypersurface it must be irreducible with isolated singularities.
Observe now that it follows that the hyperplane class $H$ on $X$ is linearly
equivalent to $\mathbf{S}+2 X_y$.
We shall compute the degree of $\mathbf{S}$ in terms of the degrees of the fiber
$X_y$ in two ways.
The first will be based on the adjunction formula on $\mathbf{S}$. Since $C_y$
is smooth we can compute its canonical class:
\begin{itemize}
\item On one hand $C_y$ is a hyperplane section of $X_y$ hence a canonical curve
of degree $d=4,6$ or $10$ hence $$\operatorname{deg}
(K_{C_y})=\operatorname{deg}(X_y). $$
\item On the other $C_y$ is a fiber of a fibration on the hypersurface
$\mathbf{S}\subset \mathbf{P}=\mathbb{P}^3$ hence $\operatorname{deg}
(K_{C_y})=\operatorname{deg}(\mathbf{S})-4$.
\end{itemize}

Next, by adjunction on $X$, we have
$\omega_\mathbf{S}=\mathbf{S}|_{\mathbf{S}}=(H-2 X_y)|_{\mathbf{S}}=H|_{\mathbf{
S}}-2 C_y.$
But, on the other hand, on $\mathbf{P}$ we have
$\omega_\mathbf{S}=(\operatorname{deg}(\mathbf{S})-4) H|_{\mathbf{S}}$.
It follows that
$\operatorname{deg}(C_y)=\frac{1}{2}(5-\operatorname{deg}(\mathbf{S}))\deg(\mathbf{S}))$ and, by
comparing with $\operatorname{deg}(C_y)=\operatorname{deg}
(K_{C_y})=\operatorname{deg}(\mathbf{S})-4$, we get a contradiction with
$\operatorname{deg}(\mathbf{S})$ being an integer. This concludes the proof for $r=3$.

Let us pass to the case $r=4$. We then have $\mathbf{P}$ is a plane and $Q^3_5$ is a
cone centered in $\mathbf{P}$ and spanned over a quadric surface $Q^0_2$. In this
case the generic fiber $X_y$ meets $\mathbf{P}$ in a finite set of points. It
follows that
$X\cap \mathbf{P}$ is a curve $\mathbf{C}$.  Consider for a generic $p\in Q_2^0$
the hyperplane $H_p$ spanned by $\mathbf{P}$ and the tangent $\tilde{T}_p Q^0_2$.
The intersection $H_p\cap Q_5^3$ decomposes into two 3-dimensional linear spaces
$L_1$ and $L_2$ such that $L_1$ contains a fiber $S$. It follows that $H_p\cap
X$ decomposes into a fiber $X_y$ of the fibration and a surface $T$ such that $T$
contains $\mathbf{C}$ and $T=X\cap L_2$.
Denote the blow up of $Q_5^3$ in $L_2$ by $\pi:\mathfrak{P}\to Q_5^3$ and the proper
transform of $X$ by the blow up by $\tilde{X}$. Since $X\cap L_2$ is a Cartier
divisor it follows that $\pi:\tilde{X}\to X$ is an isomorphism. Our aim is to prove that this is in fact impossible and in this way finish the proof. We achieve our aim 
by performing a computation on the resolution $\mathfrak{P}$ proving that $\tilde{X}$ must contain a fiber of the projection $\pi$. The computation is presented below.   

First, observe that $\mathfrak{P}\simeq
\mathbb{P}(2\mathcal{O}_{\mathbb{P}^1}(1)\oplus 3\mathcal{O}_{\mathbb{P}^1})$
admits a natural fibration $g\colon \mathfrak{P}\to \mathbb{P}^1$.
 The Chow group of $\mathfrak{P}$ is generated by the pullback $\xi=\pi^*(H)$ of the hyperplane section $H$ of $Q_5^3$  and the fiber
$h=g^*([pt])$ with the relation $\xi^5= 2 \xi^4\cdot h$ and $h^2=0$.
 It follows that the class of $X$ has a decomposition $[X]=d_{X_y} \xi^2+(d-2d_{X_y})
h\xi$.

\begin{Claim}
 We claim that one of the following holds:
 \begin{itemize}
\item $d_{X_y}=6$ and $[\tilde{X}]=6\xi^2+(d-12)h\xi$
\item $[\tilde{X}]=10 \xi^2-3h\xi$
\item $[\tilde{X}]=4 \xi^2+3h\xi$
\end{itemize} 
\end{Claim}

The claim follows from a computation with the double point formula analogous to
the proof of Theorem \ref{prop1}. Indeed, consider a section of $X$ by a generic
element from the system $\xi$. It is a surface of general type $G$ embedded in a
variety $\mathfrak{P}'\simeq \mathbb{P}(2\mathcal{O}_{\mathbb{P}^1}(1)\oplus
2\mathcal{O}_{\mathbb{P}^1})$ with generators of the Chow group
$\xi'=\xi|_{\mathfrak{P}'}$ and $h'=h|_{\mathfrak{P}'}$ such that $[G]=d_{X_y}
\xi'^2+(d-2d_{X_y}) h'\xi'$.
By the double point formula, we have $c_2(N_{G|\mathfrak{P}'})=[G]^2=2dd_{X_y}-2d_{X_y}^2$. On the
other hand, by the exact sequence
$$0\to T_G\to T_{\mathfrak{P}'}\to N_{G|\mathfrak{P}'}\to 0$$
and the Riemann--Roch theorem on $G$, we obtain another formula
$c_2(N_{G|\mathfrak{P}'})=2d_{X_y}+12d-84$. Comparing the two we prove the claim.

Before we go further, let us recall the following.
\begin{thm}[Serre's construction (for this version see {\cite[thm 1.1]{Arrondo}})]Let $X$ be a smooth variety of dimension
$\geq 3$ and $L$ a line bundle on $X$ such that
$h^2(L^{-1})=0$. Let $Y$ be a locally complete intersection subscheme of pure
codimension $2$. Then $Y\subset X$ is the zero locus of a section of a rank two
vector bundle $\mathcal{E}$ on $X$ if and only if $\omega_Y=(\omega_X\otimes L)|_Y$.
Moreover in such a case we have the
following exact sequence $$0\to \mathcal{O}_X \to \mathcal{E} \to
\mathcal{I}_{Y}\otimes
L\to 0.$$
\end{thm}

By the Serre's construction, $\tilde{X}$ is obtained as the zero locus of a
section of a rank 2 vector bundle $\mathcal{E}$ on $\mathfrak{P}$. We have
$c_1(\mathcal{E})=5\xi$ and $c_2( \mathcal{E})=[\tilde{X}]$.
Consider first the case $d_S=6$.
We see that the restriction of $\mathcal{E}$ to any fiber $F_t$ of $g$
decomposes as $\mathcal{O}_{F_t}(2)\oplus \mathcal{O}_{F_t}(3)$.
Indeed, the zero locus of the section defining $X$ restricted to $F_t$ is
Gorenstein of codimension 2 as a divisor on $X$, hence, is a complete
intersection.
This means that the restricted bundle decomposes over every fiber and this
decomposition is the same as for the generic fiber i.e.
$\mathcal{E}|_{F_t}=\mathcal{O}_{F_t}(2)\oplus \mathcal{O}_{F_t}(3)$.

It follows that
$g_* \mathcal{E}(-3\xi)$ is a line bundle and that $R^i g_*
\mathcal{E}(-3\xi)=0$ for
$i=1,2$.
This implies that $\chi( g_* \mathcal{E}(-3\xi))=\chi (\mathcal{E}(-3\xi))$. The
latter is computed by the Hirzebruch-Riemann-Roch theorem on $\mathfrak{P}$ to
be $13-d$.

Hence,
$g_* \mathcal{E}(-3\xi)=\mathcal{O}_{\mathbb{P}^1}(12-d)$. Now, by the projection
formula we have
$\mathcal{E}(-3\xi+(d-12)h)$ has a section. Since
$c_2(\mathcal{E}(-3\xi+(d-12)h))=0$ and $\mathcal{E}(-3\xi+(d-12)h)\otimes
\mathcal{O}(-D)$ has no section for non-trivial effective divisors $D$, the
section does not vanish.
It follows that we have the following exact sequence:
\begin{equation}\label{exact sequence r=4 on mathfrak P}
 0\to \mathcal{O}_{\mathfrak{P}}(3\xi-(d-12)h)\to \mathcal{E}\to
\mathcal{O}_{\mathfrak{P}}(2\xi+(d-12)h)
\end{equation}

Let us now restrict $\mathcal{E}$ to the exceptional locus $\Xi\simeq
\mathbb{P}^2\times \mathbb{P}^1$ of the blow up $\pi$ and denote $\pi''=\pi|_\Xi$
the projection onto $\mathbb{P}^2$. We observe that under the assumption that
$\tilde{X}$ contains no fiber of the projection $\pi$, we have
$\mathcal{E}'=\mathcal{E}|_E$ restricted to every fiber of the projection
$\pi''$ is a bundle of rank 2 with trivial first Chern class and either a
non-vanishing section or a section vanishing in one point i.e. it is either $\mathcal{O}_{\mathbb{P}^1}\oplus \mathcal{O}_{\mathbb{P}^1}$ or $\mathcal{O}_{\mathbb{P}^1}(-1)\oplus \mathcal{O}_{\mathbb{P}^1}(1)$. It follows that
$\pi''_*(\mathcal{E}')$ is a vector bundle of rank 2 and $R^i\pi''_*
\mathcal{E}'=0$. We compute $c_2(\pi''_*(\mathcal{E}')).$
By restricting the exact sequence \ref{exact sequence r=4 on mathfrak P} to $\Xi$,
we get
$$0\to \mathcal{O}_{E}(3\xi'-(d-12)h')\to \mathcal{E}'\to
\mathcal{O}_{E}(2\xi'+(d-12)h').$$
Taking the push-forward by $\pi''$, we have:
\begin{multline*}
   0\to \pi''_*(\mathcal{O}_{E}(3\xi'-(d-12)h'))\to \pi''_*\mathcal{E}'\to
\pi''_*\mathcal{O}_{E}(2\xi'+(d-12)h')\\
\to R^1\pi''_*(\mathcal{O}_{E}(3\xi'-(d-12)h'))\to R^1\pi''_*\mathcal{E}'\cdots
  \end{multline*}
We now observe that  $\pi''_*(\mathcal{O}_{E}(3\xi'-(d-12)h'))=0$ as
$\mathcal{O}_{E}(3\xi'-(d-12)h')$ restricted to any fiber of $\pi$ has no
nontrivial section.
Moreover, $R^1\pi''_*\mathcal{E}'=0$ as $\mathcal{E}'$ restricted to any fiber of $\pi$ is
either $2 \mathcal{O}_{\mathbb{P}^1}$ or $\mathcal{O}_{\mathbb{P}^1}(-1)\oplus
\mathcal{O}_{\mathbb{P}^1}(1)$.
Finally, by projection formula and base change
\[\pi''_*\mathcal{O}_{\Xi}(2\xi'+(d-12)h')=\mathcal{O}_{\mathbb{P}^2}(2)\otimes
H^0(\mathbb{P}^1,\mathcal{O}_{\mathbb{P}^1}(d-12))=(d-11)\mathcal{O}_{\mathbb{P}
^2}(2)\] and
\begin{multline*}R^1\pi''_*(\mathcal{O}_{\Xi}(3\xi'-(d-12)h'))=\mathcal{O}_{\mathbb{P}^2}
(3)\otimes
H^1(\mathbb{P}^1,\mathcal{O}_{\mathbb{P}^1}(12-d))=\\=\mathcal{O}_{\mathbb{P}^2}
(3)\otimes
H^0(\mathbb{P}^1,\mathcal{O}_{\mathbb{P}^1}(d-14))=(d-13)\mathcal{O}_{\mathbb{P}
^2}(3),
\end{multline*}
and we get
\[0\to\pi''_*\mathcal{E}' \to  (d-11)\mathcal{O}_{\mathbb{P}^2}(2) \to
(d-13)\mathcal{O}_{\mathbb{P}^2}(3)\to 0.\]
From the exact sequence, we compute the second Chern class of the rank 2 bundle
$\pi''_*\mathcal{E}'$:
\[c_2(\pi''_*\mathcal{E}')=\frac{1}{2}d^2-\frac{29}{2}d+108.\]
The latter is nonzero for $d\in \mathbb{Z}$.
This means that every section of  $\pi''_*(\mathcal{E}')$ vanishes in some point on
$\mathbb{P}^2$, hence, every
section of $\mathcal{E}'$ vanishes on some fiber of $\pi''$. Finally this
implies that $\tilde{X}$ contains a fiber of the blow up $\pi$ giving a contradiction which completes the proof in the case $[\tilde{X}]=6 \xi^2+(d-12)h\xi$.

The proof in the case $[\tilde{X}]=4 \xi^2+3\xi\cdot h$ is completely analogous
leading to $c_2(\pi''_*(\mathcal{E}'))=1$.
The case $[\tilde{X}]=10 \xi^2-3\xi\cdot h$ is excluded because it has negative
intersection with $[\Xi]\xi=(\xi^2-2\xi h)\xi=\xi^3-2\xi^2 h$.
\end{proof}\end{proof}

We have proved in Theorem \ref{Main theorem} that there are no smooth
Calabi--Yau threefolds of
degrees 11 and 15 contained in a quadric. For canonical surfaces of general type,
the same bounds apply only if we restrict the rank of the quadric to be at least
5.
In fact, the proof of Lemma \ref{no fibred CY in quadrics} suggests that one may construct
nodal Calabi--Yau threefolds contained in quadrics of rank 4.

Below, we recall examples of nodal Calabi--Yau threefolds and in
consequence also smooth canonical surfaces of general type of degrees 11 and 15
contained in quadrics of rank 4.

\begin{prop} \label{CY stopnia 11} Let $E_{11}=3\mathcal{O}_{\mathbb{P}^6}(1)\oplus2
\mathcal{O}_{\mathbb{P}^6}(-1)$ and let $\sigma_{11}\in
H^0((\bigwedge^2 E_{11})(1))$ be a general section.
Then $X_{11}=\operatorname{Pf}(\sigma_{11})$ is well defined and is a singular Calabi--Yau
threefold with singular locus consisting of one ordinary double point.
\end{prop}
\begin{proof}
 Since $E_{11}$ is a decomposable bundle the variety $X_{11}$ is
scheme theoretically defined by Pfaffians of a skew-symmetric matrix with
entries being general polynomials of degrees of the shape:

 \[\left(\begin{matrix}
 &3&3&1&1\\
 & &3&1&1\\
 & & &1&1\\
 & & & &-\\

 \end{matrix}\right)\]
where $-$ corresponds to the zero entry.
It is hence contained in a 2-dimensional system of quadrics defining a
cone over $\mathbb{P}^1\times \mathbb{P}^2$. This system of quadrics is generated by the $2\times 2$ minors of a $3\times 2$ matrix of linear forms. It follows that all quadrics in the system are of rank at most  4.
Let us choose a general quadric in the system and denote it $Q_5^3$. Then $Q_5^3$ is a cone with vertex a plane $\mathbf{P}$ and
spanned
over a smooth quadric surface $Q_3^0$. The quadric $Q_3^0$ has two fibrations giving two
 systems of Weil divisors on $X_{11}$. The generic element of one of them is a
complete intersection of a hyperplane and a quartic obtained as a Pfaffian of
the form

  \[\left(\begin{matrix}
 &3&3&1\\
 & &3&1\\
 & & &1\\
 
 \end{matrix}\right)\]

 the generic element of the other is a surface of degree 7 given by $2\times 2$
minors of a general matrix of the form
  \[\left(\begin{matrix}
 3&1&1\\
 3&1&1\\
  \end{matrix}\right).\]
It follows now from the end of proof of Lemma \ref{no fibred CY in quadrics}  that in the
notation of this proof we have
  $[\tilde{X}]=4 \xi^2+3\xi\cdot h$ leading to $c_2(\pi''_*(\mathcal{E}'))=1$. Hence $X_{11}$ contains a singular locus of
degree at least 1. An example with one node is given by a Macaulay2 (\cite{M2}) calculation.

\end{proof}

Let $\phi \colon 10 \mathcal{O}_{\mathbb{P}^6} \to 2
\mathcal{O}_{\mathbb{P}^6}(1)$ be a general map. Let
$E_{15}=\operatorname{ker}(\phi)\oplus \mathcal{O}(1)$.
Let $\sigma_{15}\in H^0((\bigwedge^2E_{15})(1))$ be a generic section.
\begin{prop} \label{Cy stopnia 15}
Under above assumptions the variety
$B_{15}=\operatorname{Pf}(\sigma_{15})$ is well defined and is a singular Calabi--Yau
threefold with singular locus consisting of three ordinary double points.

\end{prop}

\begin{proof}
 Similarly as in the proof for $d=11$, we prove that
$c_2(\pi''_*(\mathcal{E}'))=3$ hence $B_{15}$ contains a singular locus of
degree at least 3. The bound 3 is reached by Macaulay2 (\cite{M2}) calculation.

\end{proof}

\begin{rem}
 We saw that both Calabi-Yau threefolds  $X_{11}$ and $B_{15}$  admit two birational smooth
resolution of Picard rank 2. The extremal rays of these resolutions consist of
the small contraction of lines and one K3 fibration and one elliptic fibration.
 This means that the two birational models give all Calabi--Yau birational
 models of these varieties.

 \end{rem}
\section{Classification of Calabi-Yau threefolds contained in five-dimensional
quadrics} \label{sec classification}
In this section, we classify all nondegenerate Calabi--Yau threefolds contained in
5-dimensional quadrics.
Let $X$ be such a Calabi--Yau threefold. We know from previous sections that
$12\leq \deg(X)\leq 14$. Moreover, from the Riemann Roch theorem,
we know that any nondegenerate Calabi--Yau threefold of degree at most $13$ in $\mathbb{P}^6$
is contained in a quadric hence our classification contains all Calabi--Yau
threefolds contained
in $\mathbb{P}^6$ of degree at most $13$.

\begin{cor}\label{cor 1} A Calabi--Yau threefold $X\subset\mathbb{P}^6$ has
degree $d\geq 12$. Moreover, the degree $12$ threefold is a complete
intersection of two quadrics
and a cubic and the degree $13$ threefold is given by the $4\times  4$ Pfaffians of a
$5\times 5$ skew symmetric matrix with linear entries except one row and one
column of quadrics.
\end{cor}
\begin{proof} We deduce from the Riemann--Roch theorem that
$$\chi(\mathcal{O}_X(m))=\frac{1}{6}md(m^2-1)+7m. $$ Thus by Serre duality and
Kodaira vanishing for $d\leq 13$ we have
$h^0(\mathcal{O}_{X}(2))\leq \chi(\mathcal{O}_X(2))\leq 27 <
28=h^0(\mathcal{O}_{\mathbb{P}^6}(2))$. It follows that
$X$ has to be
contained in a quadric so the first part follows from Theorem \ref{Main
theorem}. In the case $\deg
X=12$ the threefold $X$ is contained in a pencil of quadrics. From Proposition
\ref{no CY in quadrics of rank 4}, all
the quadrics containing $X$ have rank $\geq 5$ and from the proof of Proposition
\ref{prop1} their singular locus is
always disjoint from $X$. In particular, a general quadric containing $X$ is
smooth (by the
Bertini theorem) and the intersection $\mathcal{Q}$ of two quadrics containing
$X$ is
smooth along $X$. We now apply the Lefschetz hyperplane theorem to deduce
that the Picard group of $\mathcal{Q}$ is generated by
the hyperplane section from $\mathbb{P}^6$. Since $X$ does not pass through the
singular locus of $\mathcal{Q}$, it is a Cartier divisor, hence, $X$ is cut out by a
hypersurface in $\mathbb{P}^6$ of degree three. It follows that $X$ is a
complete intersection of two quadrics and a cubic.

In the case $d=13$, we deduce from Theorem \ref{tw kw} that $X$ is not contained
in a smooth quadric. Moreover, from the proof of Theorem \ref{tw kw},
the singular locus of a quadric containing $X$ is contained in $X$.
We know also from Theorem \ref{Main theorem} that the ranks of the quadrics
containing $X$ are $\geq 5$.

Let $X\supset S\supset C$ be a general surface hyperplane section and a general
curve
linear section of $X$.
It follows that $C\subset Q_3^0$ where $Q_3^0$ is a smooth quadric in
$\mathbb{P}^4$.

Since $C\subset Q_3^0$ is subcanonical, by the Serre construction, it is the zero
locus of a section of a rank $2$ bundle
$\mathcal{E}$ on $Q_3^0$. If we now denote by $\zeta_1$ the hyperplane section and by
$\zeta_2$ the class of a line on $Q^0_3$,
we infer $c_1(\mathcal{E})=5\zeta_1$ and $c_2(\mathcal{E})=13 \zeta_2$. Thus for $\mathcal{F}=\mathcal{E}(-3)$ we compute
$c_1(\mathcal{F})=-\zeta_1$ and $c_2(\mathcal{F})=\zeta_2$.

\begin{Claim}
We claim that $\mathcal{F}$ is aCM i.e. $h^1(\mathcal{F}(n))=0$ for all
$n\in\mathbb{Z}$.
\end{Claim}

Indeed, from \cite[p.205]{AS} we know that the \textit{spinor bundle} is the only stable bundle
with these
invariants and the spinor bundle is arithmetically
Cohen--Macaulay. Hence, the claim is proved for stable $\mathcal{F}$.

Suppose now that $\mathcal{F}$ is not stable i.e. we have $h^0(\mathcal{F})\neq 0$.
Clearly, $\mathcal{F}$ cannot be decomposable since a generic section of $\mathcal{F}(3)=\mathcal{E}$ defines a
codimension 2 curve of odd degree. This means that a general section of $\mathcal{F}$
vanishes in codimension
$2$ along a curve of degree $c_2(\mathcal{F})=1$. Let us denote this line by $l$.
From the exact sequences $$ 0\to\mathcal{O}_{Q^0_3}(n)\to \mathcal{F}(n) \to
\mathcal{I}_{l|Q^0_3}(n-1)\to 0 $$
$$0\to
\mathcal{O}_{\mathbb{P}^4}(n-2)\to \mathcal{I}_l(n) \to
\mathcal{I}_{l|Q^0_3}(n)\to 0$$
 we deduce that
$h^1(\mathcal{F}(n))=h^1(\mathcal{I}_{l|Q^0_3}(n-1))=h^1(\mathcal{I}_l(n-1))$ for all $n\in
\mathbb{Z}$. Since the line $l$ is aCM, we
have $h^1(\mathcal{F}(n))=0$ for all $n\in \mathbb{Z}$. This proves the claim.

Applying the claim to the cohomology of the exact sequence
$$ 0\to\mathcal{O}_{Q^0_3}(n)\to \mathcal{F}(n+3) \to
\mathcal{I}_{C|Q^0_3}(n+5)\to 0, $$ it
follows that
 $h^1(\mathcal{I}_{C|Q^0_3})(n)=0$ for all $n\in\mathbb{Z}$. Now, from the exact
sequence $$0\to
\mathcal{O}_{\mathbb{P}^4}(n-2)\to \mathcal{I}_C(n) \to
\mathcal{I}_{C|Q^0_3}(n)\to 0,$$
 we infer that $h^1(\mathcal{I}_C(n))=0$ for all $n\in\mathbb{Z}$.

Consider now the following exact sequences:
\begin{equation}\label{exactXS}
 0\rightarrow \mathcal{I}_X(n-1)\rightarrow \mathcal{I}_X(n)\rightarrow
\mathcal{I}_S(n)\rightarrow 0
\end{equation}
\begin{equation}\label{exactSC}
0\rightarrow \mathcal{I}_S(n-1)\rightarrow \mathcal{I}_S(n)\rightarrow
\mathcal{I}_C(n)\rightarrow 0.
\end{equation}
We know that $h^2(\mathcal{I}_X(n))=0$ and, from Proposition \ref{pro linearly normal}, that
$h^1(\mathcal{I}_X(1))=0$. From the long exact sequences of cohomology
corresponding to \ref{exactXS}, we deduce that
$h^1(\mathcal{I}_S(1))=0$.
Since $h^1(\mathcal{I}_C(n))=0$ for all $n\in\mathbb{Z}$ it follows by induction
that the long exact sequence constructed from (\ref{exactSC}) implies $h^1(\mathcal{I}_S(n))=0$. Then, by the sequence (\ref{exactXS}),
we have
$h^1(\mathcal{I}_X(n))=0$, so $X\subset \mathbb{P}^6$ is aCM.
The assertion follows from  Corollary \ref{Cor from Buchsbaum CY}.
\end{proof}

To complete the classification of Calabi--Yau threefolds contained in quadrics, we
lack only the consideration of Calabi--Yau threefolds of degree $14$.
For this, let us analyze more precisely Calabi--Yau threefolds contained in smooth 5-dimensional quadrics.
Let $X\subset Q^0_5$ be a Calabi--Yau threefold. Since $X$ is subcanonical,
from the Serre construction we
obtain that $X$ is the zero locus of a section of a rank $2$ vector bundle $\mathcal{E}$
i.e.
\begin{equation}\label{eq a} 0\to \mathcal{O}_{Q^0_5}\to \mathcal{E}\to
\mathcal{I}_{X|Q^0_5}(c_1(\mathcal{E}))\to 0.\end{equation}
Since $K_X=0$ we obtain $c_1(\mathcal{E})=5\zeta_1$ where $\zeta_1\subset Q^0_5$
is the
hyperplane section of $Q_5^0$. Moreover, we have $c_2(\mathcal{E})=\frac{1}{2}\deg(X)
\zeta_2$, where $\zeta_2$ is a general codimension two linear section of $Q^0_5$.
Finally
$H^i(\mathcal{E}(k))=H^i(\mathcal{I}_{X|Q^0_5}(c_1(\mathcal{E})+(5+k))$.

It follows, by applying the Beilinson type spectral sequence as in \cite{S}, that the only arithmetically Cohen-Macaulay bundles
(i.e with vanishing intermediate cohomology) on
$Q^0_n$ with $5\geq n\geq3$ are direct sums of line bundles and spinor
bundles.
Moreover, in \cite[Thm.~1.7]{BHVV}, it is proved that arithmetically Buchsbaum rank $2$ vector
bundles on $Q^0_5$, that are not aCM, are twists of Cayley bundles.
Recall that a \textit{Cayley bundle} $\mathcal{C}$ is defined in \cite{Ot} by the exact sequences
$$0\to \mathcal{O}_{Q^0_5}\to \mathcal{S}\to G\to 0,$$
$$0\to \mathcal{O}_{Q^0_5}\to G(1)\to \mathcal{C}(1)\to 0,$$
where $G$ is obtained by choosing a general section of the spinor bundle $\mathcal{S}$.
It means that Calabi-Yau threefolds given as zero loci of sections of
arithmetically Buchsbaum bundles on $Q^0_5$ are either degree 12 complete
intersections or zero loci of twists of Cayley bundles.
We saw in Corollary \ref{cor 1} that in degree 12 the complete intersections
form a unique family of examples.
We prove a similar result in the case of degree 14.
Let us first construct two families of Calabi--Yau threefolds of degree 14.

\begin{prop} \label{przyklad stopnia 14 w gladkiej kwadryce}
The zero locus of a generic section of the homogenous bundle $\mathcal{C}(3)$ is a
Calabi--Yau threefold of degree $14$ contained in $Q^0_5$.

\end{prop}
\begin{proof}
We have $c_1(\mathcal{C}(k))=2k-1$ thus to obtain a Calabi--Yau threefold we put $k=3$.
Then $c_2(\mathcal{C}(3))=14$. We know from \cite[thm.~3.7]{Ot} that $\mathcal{C}(2)$ is globally generated. 
We deduce that the zero locus of a general section $\mathcal{C}(3)$ is a smooth
Calabi--Yau threefold of degree $14$.
\end{proof}

The following family of Calabi-Yau threefolds of degree 14 was introduced in
\cite{Be}.
\begin{prop}\label{przyklad stopnia 14 bertin}
 Let $E_{14}=\Omega^1_{\mathbb{P}^6}(1)\oplus \mathcal{O}_{\PP^6}(1)$. Then the degeneracy locus
$D_6(\phi)$ of a generic skew symmetric map $\phi \colon E_{14}^*(-1)\to E_{14}$ is a
Calabi-Yau threefold of degree 14.
 
\end{prop}
\begin{proof} Let $B_{14}$ be a general variety obtained by this construction.
By the Bertini type theorem for Pfaffian subvarieties given in \cite[\textsection 3]{O} to
prove that $B_{14}$ is a smooth threefold it is
enough to prove that the bundle $\bigwedge^2 E_{14}(1)$ is globally generated.
Now, $\bigwedge^2 E_{14}(1)= \Omega_{\PP^6}^2(3)\oplus\Omega_{\PP^6}^1(3)$.
The Euler sequence gives us:
$$0\rightarrow \Omega_{\PP^6}^1(1)\rightarrow 7\mathcal{O}_{\PP^6}\rightarrow
\mathcal{O}_{\PP^6}(1)\rightarrow 0.$$
Taking its second and third wedge powers we have:
$$0\rightarrow \Omega_{\PP^6}^2(2)\rightarrow {\binom{7}{2}}
\mathcal{O}_{\PP^6}\rightarrow \Omega^1(2)\rightarrow 0$$
and
$$0\rightarrow \Omega_{\PP^6}^3(3)\rightarrow {\binom{7}{3}} \mathcal{O}_{\PP^6}\rightarrow
\Omega_{\PP^6}^2(3)\rightarrow 0.$$
It follows that $\Omega_{\PP^6}^1(2)$ and $\Omega_{\PP^6}^2(3)$ are globally generated, hence,
$\Omega_{\PP^6}^1(3)$ and $\Omega_{\PP^6}^2(3)\oplus \Omega_{\PP^6}^1(3)$ are also globally generated.
It follows that $B_{14}$ is a Pfaffian variety associated to the
bundle $E_{14}=\Omega^1_{\mathbb{P}^6}(1)\oplus \mathcal{O}_{\PP^6}(1)$, with $t=1$.
The vanishing of the dualizing sheaf is then given by the adjunction formula
(\ref{pfaffian adjunction}) for Pfaffian varieties and the vanishing of the
cohomology by the Pfaffian sequence \ref{pfaffian sequence}.
\end{proof}

 We shall denote by $\mathfrak{C}_{14}$ and $\mathfrak{B}_{14}$ the families of Calabi--Yau threefolds obtained 
 in Proposition \ref{przyklad stopnia 14 w gladkiej kwadryce} and Proposition \ref{przyklad stopnia 14 bertin} respectively.
Observe that a generic $B_{14}\in \mathfrak{B}_{14}$ is contained in a quadric.
Indeed, among the equations defining $B_{14}$ we have a section of the bundle
$(\bigwedge^6(\Omega^1_{\mathbb{P}^6}(1)))(3)=\mathcal{O}_{\mathbb{P}^6}(2)$
corresponding to the Pfaffian of a map $(\Omega^1_{\mathbb{P}^6}(1))^*(-1)\to \Omega^1_{\mathbb{P}^6}(1)$.

We can now classify all Calabi--Yau threefolds of degree 14 contained in a 5-dimensional quadric.

\begin{cor} \label{stopnia 14 w kwadryce}
If $X$ is a Calabi--Yau threefold of degree 14 contained in a smooth quadric $Q^0_5$ then $X$ is obtained as
the zero locus of a section of the twisted Cayley bundle $\mathcal{C}(3)$.
Moreover, if $X$ is a Calabi--Yau threefold contained in any quadric it is given
by the Pfaffian construction applied to
$E_{14}=\Omega^1_{\mathbb{P}^6}(1)\oplus \mathcal{O}_{\PP^6}(1).$
\end{cor}

\begin{proof}
Let $X$ be a Calabi--Yau threefold of degree 14 contained in the smooth
quadric $Q^0_5$.
By the Serre construction, we know that $X$ is given as the zero locus of a
section of a rank 2 vector bundle $\E$ such that
$c_1(\E(-3))=-\zeta_1$ and $c_2(\E(-3))=\zeta_2$.
It follows from the main theorem in \cite{Ot} that $\E(-3)$
is a Cayley bundle if we assume it is stable.
To prove that $\E(-3)$ is stable is the same as to prove that
$h^0(\E(-3))=h^0(\mathcal{I}_{X|Q^0_5}(2))=0$ (see (\ref{eq a})).
To prove that $X$ is not contained in a second quadric (i.e. that $h^0(\mathcal{I}_{X|Q^0_5}(2))=0$), we argue as in the proof
of Corollary \ref{cor 1}.
More precisely, if $\E(-3)$ is not stable then its general section defines a degree 2
threefold $W$ contained in $Q^0_5$. Since $Q^0_5$ contains no 3-dimensional linear
space, $W$ must be a 3-dimensional quadric.
Hence, $W$ is aCM implying $\E$ is aCM and in consequence $X$ is aCM. Thus $X$ is
not contained in any quadric and this gives a contradiction.

To prove the second part, we first show that $h^1(\mathcal{I}_X(k))=0$ for
$k\neq 2$ and $h^1(\mathcal{I}_X(2))=1$.
From the long cohomology exact sequence obtained from \ref{exactXS} and
\ref{exactSC}, we infer $h^1(\mathcal{I}_X(n))=h^1(\mathcal{I}_Z(n))$ where
$C\subset \mathbb{P}^4$ is a codimension $2$ linear section of $X\subset
\mathbb{P}^6$.
But we know that $X$ is contained in a quadric of rank $\geq 5$ thus $C$ is
contained in a smooth quadric. It follows that $C$ is given as the zero locus of a section of a rank $2$
vector bundle $\mathcal{B}(3)$ such that $c_1(\mathcal{B})=-\zeta_1$ and $c_2(\mathcal{B})=\zeta_2$. It
follows that the zero locus of a general section of $\mathcal{B}$
is a sum $t$ of two skew lines. We find that the
cohomologies of $h^1(\mathcal{I}_{t|\mathbb{P}^3}(k))$ are $0$ for $k\neq 2$ and
$1$ for $k=2$. Since all nonzero elements of the Hartshorne--Rao module of $X$ are of the same weight, it follows that $X$ is quasi-Buchsbaum. 
Hence, we deduce from Theorem \ref{Buchsbaum CY} that $X=\Pf (\sigma)$ for some $\sigma \in H^0(\bigwedge^2 E_{14}(1))$ where 
$E_{14}=\Omega_{\PP^6}^1(1)\oplus \mathcal{O}_{\mathbb{P}^6}(1)$.
We thus get the assertion.
\end{proof}

As a direct consequence of the above, we have the following.
\begin{cor}\label{C14 contained in B14}
The family $\mathfrak{C}_{14}$ is contained in $\mathfrak{B}_{14}$ as a dense
subset.
\end{cor}

\section{Classification of degree 14 Calabi-Yau threefolds in $\mathbb{P}^6$}\label{sec class degree 14}
The aim of this section is the classification of all degree $14$ Calabi--Yau
threefolds in $\mathbb{P}^6$.

From the Riemann-Roch theorem, we deduce that if $X\subset\mathbb{P}^6$ is a
Calabi-Yau threefold of degree $14$ then
$h^1(\mathcal{I}_X(2))=h^0(\mathcal{I}_X(2))$ and
$h^1(\mathcal{I}_X(3))+7=h^0(\mathcal{I}_X(3))$ (recall that
$h^0(\mathcal{I}_X(1))= 0$).
We have two possibilities:
\begin{itemize}
 \item $X$ is contained in a quadric,
 \item $X$ is not contained in any quadric.
\end{itemize}
The first case is solved by Corollary \ref{stopnia 14 w kwadryce}.
Assume, hence, that $X$ is not contained in any quadric. Then there is an at least
$7$-dimensional space of cubics vanishing along $X$. We, moreover, claim the following.
\begin{prop}\label{l1} The singular locus of a generic cubic from
$H^0(\mathcal{I}_X(3))$ has dimension $\leq 1$.
\end{prop}
Before we pass to the proof of Proposition \ref{l1}, let us prove the following.
  \begin{lemm}\label{l3} Suppose that a Calabi--Yau threefold in $\mathbb{P}^6$
is not contained in any quadric. Let $X\supset S \supset C \supset F$ be general
linear sections of $X$ of codimension 1,2,3 respectively.
Then, we have
$h^0(\mathcal{I}_X(3))=h^0(\mathcal{I}_S(3))=h^0(\mathcal{I}_C(3))=h^0(\mathcal{
I}_F(3))$ and
$0=h^0(\mathcal{I}_X(2))=h^0(\mathcal{I}_S(2))=h^0(\mathcal{I}
_C(2))=h^0(\mathcal{I}_F(2))$.
  \end{lemm}
  \begin{proof} First, we have $h^i(\mathcal{I}_{X}(k))=0$ for $i=2,3$ and
$k\in\mathbb{Z}$. Furthermore, we assumed $0=h^0(\mathcal{I}_X(m))=h^1(\mathcal{I}_X(m))$
for $m=1,2$.
  From the exact sequence
$$0\to\mathcal{I}_X\to\mathcal{I}_X(1)\to\mathcal{I}_S(1)\to 0$$ tensorized by
line bundles, we infer: 
\begin{itemize}
 \item $h^i(\mathcal{I}_S(m))=0$ for $i=0,1,2$ and $m=0,1$;
\item $h^i(\mathcal{I}_S(2))=0$ for $i=1,2$; 
\item $H^0(\mathcal{I}_X(3))\simeq H^0(\mathcal{I}_S(3))$.
\end{itemize}

 We conclude by repeating this
procedure for a hyperplane section and a codimension $2$ linear section.
  \end{proof}
We can now pass to the proof of Proposition \ref{l1}
\begin{proof}[Proof of Proposition \ref{l1}] We keep the notation from Lemma
\ref{l3}.
By Lemma \ref{l3}, it is enough to prove that the general cubic hypersurface in
$\mathbb{P}^4$ containing $C$ is smooth.
For this, we need to study the scheme-theoretic intersection of all cubics
containing $C$ or more generally the scheme-theoretic intersection of cubics
containing $X$.
Let $\Lambda$ be the scheme-theoretic intersection of the cubics
containing $X$.
Then either $X$ is a component of $\Lambda$ (necessarily of multiplicity $1$ for
degree reason), or $\Lambda$ has a component
$Y\subset\mathbb{P}^6$ of codimension $2$ such that $X\subset Y$.

In the first case, we use the excess
 intersection formula. We choose $C_1,\dots,C_6$ generic cubics containing $X$.
 Then we find $(C_1.\dots.C_6)^X$ the equivalences of $X$ in the intersection
$C_1.\dots.C_6$.
  Indeed, from \cite[prop.~9.1.1]{Ful}, we infer
$$(C_1.\dots.C_6)^X=(\sum_{i=1}^{6}c(N_{C_i|\mathbb{P}^6}|_X) c(T_{\mathbb{P}^6}
|_X)c(T_X))_0=c_3+15h^3-11h.c_2=3^6-1$$ that is an element from $A_0(X)$.
  The equivalences of all the distinguished components \cite[\& 6.1]{Ful} of
$C_1\cap\dots \cap C_6$ sum up to $C_1.\dots.C_6$. Moreover, by the refined
Bezout theorem \cite[thm.~12.3]{Ful}, the equivalences of the distinguished
components are positive numbers of degree bigger than the degree of this
component and all the irreducible components of $\bigcap C_i$ are among the
distinguished components.
  It follows that there can be at most one such component whose support is not
contained in $X$ and this component have degree $1$ with multiplicity $1$ in the
intersection of $C_i$ for $1\leq i\leq 6$. Since the equivalences of linear
spaces of dimension $\geq 1$ in the intersection of cubics are bigger than 2, we
deduce that there can be only one point outside $X$ in the intersection
$\bigcap_{i=1}^6 C_i$.

 Let us now consider in this case a general codimension $2$ linear section $C$
of $X\subset
\mathbb{P}^6$. We can choose it not to pass through the additional point. We
then have that $C$
is set theoretically defined by the cubics containing it. Moreover, $C$ is a
component of
multiplicity one of the scheme theoretical intersection of these cubics. It
follows that $C$ is
scheme theoretically defined by the cubics containing it outside possibly a
finite set $P$ of point on $C$.
Then we know from \cite[thm 2.1]{DH} that the generic cubic containing $C$ is
smooth outside the
set $P$.
  Suppose that there is a point from $P$ such that the generic cubic containing
$C$ is
singular in it. Since the choice of the codimension 2 linear section giving $C$
was generic, it means that the intersection of the
singular loci of all cubics containing $X$ contains a surface $U$. Then each
cubic containing $X$ must contain the secant variety $\mathfrak{s}_2(U)$ of the surface
$U$.
As we already saw that set theoretically the intersection defines $X$ plus
possibly one point we infer $\mathfrak{s}_2(U)\subset X$. Since $X$ is a Calabi--Yau
threefold, we have $\mathfrak{s}_2(U)\neq X$ and, hence, $\mathfrak{s}_2(U)$ must be a surface.
It follows that $U$ is a plane contained in $X$. But then the fiber of the
projection from $U\subset \mathbb{P}^6\to \mathbb{P}^3$ intersects the cubics
containing $X$ in linear spaces outside $U$. Thus either $X$ is rational or $X$
is covered by lines. This is a contradiction in any case.

  Assume now that $\Lambda$ has a component $Y\subset\mathbb{P}^6$ of
codimension $2$ such that $X\subset Y$.
We still keep the notation from Lemma \ref{l3}.
Observe that, in this case, all the cubics in $\mathbb{P}^3$ containing $F$
contain also a fixed curve
$D\subset \mathbb{P}^3$ being the codimension $3$ linear section of $Y$. Moreover,
since $X\subset Y$, we have $F\subset D$.
By Lemma \ref{l3}, the projective linear system of cubics containing $D$ is of
dimension at least 6. Moreover, since $F$ is not contained in any quadric,
the restrictions of these cubics to a general hyperplane in $\mathbb{P}^3$ form
a projective linear system of cubics on $\mathbb{P}^2$ of the same dimension 6.
It follows that the intersection of $D$ with a generic hyperplane is a scheme of
length at most 3. Hence, $D$ is a curve of degree 3 in $\mathbb{P}^3$.
Such a curve is either contained in a hyperplane or in a quadric hypersurface.
The latter is a contradiction with $F\subset D$ and $h^0(\mathcal{I}_F(2))=0$.
\end{proof}

\begin{cor} A Calabi--Yau threefold of degree $14$ in $\mathbb{P}^6$ that is not
contained in any quadric is defined by the $6\times 6$ Pfaffians of a skew
$7\times 7$ matrix with linear entries.
\end{cor}
\begin{proof}
Let us consider a generic codimension $2$ linear section $C$ of $X\subset
\mathbb{P}^6$.
From Proposition \ref{l1}, the curve $C$ is contained in a smooth cubic
threefold $W$.
Let $H$ be the class of the hyperplane section of $W$. Since $K_C=2H$ and
$K_T=-2H$, we deduce that $C\subset W$ is subcanonical.
From the Kawamata-Viehweg vanishing theorem, we have $h^1(\mathcal{O}_W(-4H))=0$
and $h^2(\mathcal{O}_T(-8H))$. Thus we can apply the Serre construction and find
a rank $2$
vector bundle $\mathcal{E}$ on $W$ with a section vanishing along $C\subset W$.
More precisely, we obtain an exact sequence \begin{equation}\label{e1}
0\to\mathcal{O}_W\to \mathcal{E}\to \mathcal{I}_{C|W}(4H)\to 0 \end{equation} and
$c_1(\mathcal{E})=4H$
and $H\cdot c_2(\mathcal{E})=14$. We compute that $c_1(\mathcal{E}(-2))=0$ and $H\cdot c_2(\mathcal{E}(-2))=2$.
Moreover, $\mathcal{E}$
is stable if and only if  $h^0(\mathcal{E}(-2))=0$.
From the exact sequence (\ref{e1}) tensorized by $\oo(-2H)$, we deduce that
$h^0(\mathcal{E}(-2))\neq 0$ implies that $h^0(\mathcal{I}_{C|W}(2))\neq 0$. This contradicts the fact that $C$ is not contained in any quadric. Thus $\mathcal{E}$ is stable.
By Serre duality, $0=h^2(\mathcal{E}(-2))=h^0(\mathcal{E}(-3))$. It is proven in \cite[thm.~2.4]{Dr}
that if $\mathcal{E}(-2)$ is stable then $h^i(\mathcal{E}(-1-i))=0$ for $i\geq 1$. It follows that
$h^1(\mathcal{E}(n))=0$ for $n\in \mathbb{Z}$. Now, from the long exact sequence of cohomology associated
to the exact sequence (\ref{e1}), we infer $h^1(\mathcal{I}_{C|W}(4-k))\leq
h^2(\mathcal{O}_W(-k))=h^1(\mathcal{O}_W(-2+k))$. From the Kodaira vanishing
theorem, the last number is $0$ so $h^1(\mathcal{I}_{C|W}(k)))=0$. Next, from
the long exact sequence of cohomology associated to the sequence $$0\to
\mathcal{O}_{\mathbb{P}^4}(-3)\to \mathcal{I}_{C|\mathbb{P}^4}\to
\mathcal{I}_{C|W}\to 0,$$ we obtain $$h^1(\mathcal{I}_{C|\mathbb{P}^4}(k)))=0.$$
Finally, arguing as in Lemma \ref{l3} we infer
$h^1(\mathcal{I}_{X}(k))=0$ for $k\in \mathbb{Z}$ thus $X\subset \mathbb{P}^6$
is aCM. The assertion follows from Corollary \ref{Cor from Buchsbaum CY}.
\end{proof}
\section{Classification up to deformations} \label{sec examples}
By the classification of Section \ref{sec classification} for $d\leq 13$, the Hilbert scheme of Calabi--Yau threefolds of degree $d$ has a unique irreducible component.
In this section, we prove that the statement is also valid for $d=14$.
To do this, we compare the families of Calabi--Yau threefolds of degree $14$ in $\mathbb{P}^6$ appearing in the classification above.
We show that all such varieties are smooth degenerations of the family of
Calabi--Yau threefolds of degree $14$ defined by the $6\times 6$ Pfaffians of an
alternating $7\times 7$
matrix, hence, they are in the same component of the Hilbert scheme.

Denote by $\mathfrak{T}_{14}$ the family of Calabi-Yau threefold defined by $6\times 6$ Pfaffians of
$7\times7$ matrices of linear forms.
We observed that the general element $T_{14}\in \mathfrak{T}_{14}$ is not contained
in any quadric.
On the other hand, we have two families $\mathfrak{C}_{14} \subset \mathfrak{B}_{14}$ of Calabi--Yau threefolds contained in a quadric.

We compute the dimension of $\mathfrak{C}_{14}$ using Bott formula.
From the computation and Corollary \ref{C14 contained in B14}, it follows that the dimension of the component of the
Hilbert scheme of Calabi--Yau threefolds of degree 14 in $\mathbb{P}^6$
containing $\mathfrak{B}_{14}$ is bigger than the dimension of
$\mathfrak{B}_{14}$.
In fact, we prove the following.

\begin{thm}\label{smooth degeneration of degree 14} For a generic Calabi-Yau
threefold $B_{14}$ belonging to
$\mathfrak{B}_{14}$, there exists a smooth morphism
$\mathbb{P}^6\times \Delta \supset
\mathcal{X}\rightarrow \Delta$ to the complex disc $\Delta$ such that
for $\lambda\neq 0$ the fiber $\mathcal{X}_{\lambda}$ is a smooth subvariety in
$\mathbb{P}^6$ belonging to the family $\mathfrak{T}_{14}$ whereas the central
fiber $\mathcal{X}_{0}=B_{14}$.
\end{thm}

The theorem is a straightforward consequence of the Euler sequence and the
following Proposition \ref{degenerations of pfaffians}.

To formulate the proposition let us start with two vector bundles $E$, $F$ of ranks $2u$ and $2u+1$ on
$\mathbb{P}^6$ forming an exact
sequence.
\begin{equation}
 0\rightarrow E\rightarrow F\rightarrow \mathcal{O}_{\mathbb{P}^6}(1)\rightarrow
0
\end{equation}
 Consider the exact sequence relating wedge squares of the bundles tensorized by
$\mathcal{O}_{\mathbb{P}^6}(1)$:
  \begin{equation}
   0\rightarrow (\bigwedge^2 E)(1)\rightarrow (\bigwedge^2 F)(1)\rightarrow
E(2)\rightarrow 0,
  \end{equation}
and its associated cohomology sequence:
  \begin{equation}
   0\rightarrow H^0((\bigwedge^2 E)(1))\xrightarrow{\eta} H^0((\bigwedge^2
F)(1))\xrightarrow{\delta} H^0(E(2))\rightarrow H^1((\bigwedge^2 E)(1)).
  \end{equation}
 Assume that the map $\delta$ is surjective.
 Assume moreover that the Pfaffian varieties $X_{\sigma}$ and $X_{\sigma'}$ associated
to generic sections $\sigma\in H^0((\bigwedge^2 E)(1)\oplus
\mathcal{O}_{\mathbb{P}^6}(1))$
 and $\sigma'\in H^0((\bigwedge^2 F)(1))$ are irreducible of codimension 3 as
expected.

\begin{prop}\label{degenerations of pfaffians}
 Let $E$,$F$ be vector bundles as above. Then for a generic section $s\in
H^0((\bigwedge^2 (E \oplus \mathcal{O}_{\mathbb{P}^6}(1)))(1))$, there exists a
family of
 sections $\sigma'_{\lambda}\in H^0((\bigwedge^2 F)(1))$ parametrized by $\lambda\in
\mathbb{C}\setminus\{0\}$ such that the family
 \begin{displaymath}
 X_\lambda=\begin{cases}\operatorname{Pf}(\sigma'_\lambda) \quad \text{ for}
\lambda\neq 0\\

            \operatorname{Pf}(\sigma) \quad
\text{ for} \lambda= 0
           \end{cases}
\end{displaymath}
is a flat family of 3-dimensional subvarieties of $\mathbb{P}^6$.
\end{prop}

\begin{proof}
 Let $\sigma\in H^0((\bigwedge^2 (E \oplus \mathcal{O}_{\mathbb{P}^6}(1)))(1))$. We
have $\sigma=(\sigma_1,\sigma_2)$ with $\sigma_1\in H^0(\bigwedge^2 E)(1))$ and $\sigma_2\in H^0(E(2))$.
By the assumption made, there exists a pair
 $(\sigma'_1,\sigma'_2)$ of sections $\sigma_1,\sigma_2 \in H^0(\bigwedge^2 F)(1)$ such that $\eta
(\sigma_1)=\sigma'_1$, $\delta(\sigma'_2)=\sigma_2$. Let $\sigma'_{\lambda}=\sigma'_1 +\lambda \sigma'_2\in
H^0((\bigwedge^2 F)(1))$.

Consider the variety $\Delta\times \mathbb{P}^6$ with projection $\pi_{\Delta}$
and $\pi_{\mathbb{P}^6}$.
Consider now the subscheme $\tilde{\mathcal{X}}\subset \Delta\times
\mathbb{P}^6$ defined by the vanishing of the section
$\Delta\times \mathbb{P}^6\ni (\lambda,x)\mapsto
((\lambda,x),\sigma'_{\lambda}(x)^{(\wedge r)}) \in \pi_{\mathbb{P}^6}^*((\bigwedge^2u
F)(u)),$
and consider $\mathcal{X}$ its irreducible component dominating $\Delta$. By
\cite[Prop. 9.7]{Har}, the family $\pi_{\Delta}\colon \mathcal{X}\to \Delta$ is
flat. We clearly see that the fiber $\mathcal{X}_{\lambda}$ for $\lambda\neq 0$
is $X_{\lambda}$.
It is, hence, enough to prove that the fiber $\mathcal{X}_0$ over $0\in \Delta$ is
equal to $X_0$

We have $(\sigma'_1+\lambda \sigma'_2)^{\wedge u}=(\sigma'_1)^{\wedge u}+\lambda\gamma_1$ for some $\gamma_1 \in H^0((\bigwedge^{2u} F)(u))$. It
follows that $\mathcal{X}_0$ is contained in the zero locus
$Z((\sigma_1')^{\wedge u})$ of the section
$(\sigma_1')^{\wedge u}\in H^0(\bigwedge^{2u} F)$. Moreover, by the exact sequence:
\[0\rightarrow (\bigwedge^{2u} E)(u)\rightarrow (\bigwedge^{2u}
F)(u)\xrightarrow{\psi} (\bigwedge^{2u-1} E)(u+1)\to 0,\]
we infer $Z((\sigma_1')^{\wedge u})=Z(\sigma_1^{\wedge u})$.
Furthermore, knowing that
\[\psi((\sigma'_1+\lambda \sigma'_2)^{\wedge u})= \lambda \sigma_1^{\wedge u-1}\wedge \sigma_2+
\lambda^2 \gamma_2 \] for some $\gamma_2\in H^0((\bigwedge^{2u} F)(u))$ 
and that
\[Z((\sigma'_1+\lambda \sigma'_2)^{\wedge u})\subset Z(\psi((\sigma'_1+\lambda \sigma'_2)^{\wedge
u})),\]
we see that
$\mathcal{X}_0\subset Z(\sigma_1^{\wedge u-1}\wedge \sigma_2)$.
Putting everything together, we get:  $\mathcal{X}_0$ is contained in the zero
locus of 
\[(\sigma_1^{\wedge u},\sigma_1^{\wedge u-1}\wedge \sigma_2)\in
H^0((\bigwedge^{2u}(E\oplus\mathcal{O}_{\mathbb{P}^6}(1)))(u)=H^0((\bigwedge^2u
E)(u))\oplus H^0(\bigwedge^{2u-1} E(u+1).\]
The latter zero locus is precisely the variety
$X_0=\operatorname{Pf}((\sigma_1,\sigma_2))$.
It follows that $ \mathcal{X}_0\subset X_0$. But, since
$\pi_{\Delta}|_{\mathcal{X}}\colon \mathcal{X}\to \Delta $ is flat, we know that
$\mathcal{X}_0$ is of codimension 3 in $\mathbb{P}^6$ which implies by
assumption on $X_0$ that $\mathcal{X}_0= X_0$.\end{proof}
\begin{proof}[Proof of Theorem \ref{smooth degeneration of degree 14}]
The Euler sequence gives:
\[0\rightarrow \Omega^1_{\mathbb{P}^6}(1)\rightarrow
7\mathcal{O}_{\mathbb{P}^6}\rightarrow \mathcal{O}_{\mathbb{P}^6}(1)\rightarrow
0.\]
By Proposition \ref{degenerations of pfaffians}, for any $B_{14}\in
\mathfrak{B}_{14}$ we obtain a flat family of manifolds defined as Pfaffians
associated to the bundle $7\mathcal{O}$ degenerating to $B_{14}$.
Since smoothness is an open condition in flat families, we conclude that any
smooth $B_{14}$ is a smooth degeneration of a family of smooth Calabi-Yau
threefolds from $\mathfrak{T}_{14}$.
\end{proof}

\begin{rem} Similarly, we prove using Proposition \ref{degenerations of pfaffians}
that any Calabi--Yau threefold $B_{15}$ from Proposition \ref{Cy stopnia 15} is a flat degeneration of the family of Calabi--Yau
threefolds of degree 15 defined as Pfaffians of the bundle
$\Omega_{\mathbb{P}^6}^1(1)\oplus 3\mathcal{O}_{\mathbb{P}^6}$.
\end{rem}

\vskip2cm

\begin{minipage}{15cm}
Institut f\"ur Mathematik\\
Mathematisch-naturwissenschaftliche Fakult\"at\\
Universit\"at Z\"urich, Winterthurerstrasse 190, CH-8057 Z\"urich\\
\end{minipage}

\begin{minipage}{15cm}
 Department of Mathematics and Informatics,\\ Jagiellonian
University, {\L}ojasiewicza 6, 30-348 Krak\'{o}w, Poland.\\
\end{minipage}

\begin{minipage}{15cm}
Institute of Mathematics of the Polish Academy of Sciences,\\
ul. \'{S}niadeckich 8, P.O. Box 21, 00-956 Warszawa, Poland.\\
\end{minipage}

\begin{minipage}{15cm}
\emph{E-mail address:} grzegorz.kapustka@uj.edu.pl\\
\emph{E-mail address:} michal.kapustka@uj.edu.pl
\end{minipage}
\end{document}